\documentclass[a4paper,11pt]{article}

\usepackage{amssymb}
\usepackage{amstext}
\usepackage{amscd}
\usepackage{makeidx}
\usepackage{amsmath,amsthm}
\usepackage[dvips]{graphicx}
\usepackage{fancyhdr}
\usepackage{fancybox}
\usepackage[latin1]{inputenc}

\newcommand{\e}{\varepsilon}
\newcommand{\Ep}{\mathcal{E}}
\newcommand{\N}{\mathbb{N}}
\newcommand{\Z}{\mathbb{Z}}

\newcommand{\R}{\mathbb{R}}

\newcommand{\E}{\mathbb{E}}

\newcommand{\1}{1\! \mathrm{l}}

\newcommand{\y}{\mathrm{y}} 
\newcommand{\te}{\tilde{\varepsilon}}

\newcommand{\Pro}{\mathbb{P}}
\newcommand{\Hc}{\mathcal{H}}

\def\limiteN{\renewcommand{\arraystretch}{0.5}
\begin{array}[t]{c}\xrightarrow{\; \qquad \; }\\
{\scriptstyle N \rightarrow \infty} \end{array}
\renewcommand{\arraystretch}{1}}

\def\limiteas{\renewcommand{\arraystretch}{0.5}
\begin{array}[t]{c}\xrightarrow{\; \nu-a.s. \;}\\
{\scriptstyle N\rightarrow \infty} \end{array}
\renewcommand{\arraystretch}{1}}

\newtheorem{theorem}{Theorem}
\newtheorem{definition}{Definition}

\newtheorem{remark}{Remark}
\newtheorem{lemma}{Lemma}

\newtheorem{example}{Example}
\begin{document}

\title{Aggregation of weakly dependent doubly stochastic
processes}

\author{Lisandro J. Ferm\'{\i}n\\ {\scriptsize Universit\'e Paris-Sud \
}}  \date{} \maketitle

\abstract{The aim of this paper is to extend the aggregation convergence results
given in \cite{Dacunha&Fermin.L.Notes, Dacunha&Fermin.aggreglinear} to doubly stochastic linear and nonlinear processes with weakly dependent innovations. First, we
introduce a weak dependence notion for doubly stochastic processes,
based in the weak dependence definition given by Doukhan and
Louhichi, \cite{Doukhan&Louhichi}, and we exhibe several models satisfying this
notion, such as: doubly stochastic Volterra processes and doubly
stochastic Bernoulli scheme with weakly dependent innovations.
Afterwards we derive a central limit theorem for the partial
aggregation sequence considering weakly dependent doubly stochastic
processes. Finally, show a new SLLN for the covariance function of
the partial aggregation process in the case of doubly stochastic
Volterra processes with interactive innovations.
}

\thispagestyle{empty}
\section{Introduction}

\hspace{0.4cm} The aggregation of doubly stochastic processes with interactive
innovations has been little studied, because dependent innovations
induce a dependency structure for the doubly stochastic
processes which can be difficult to study, especially in the case of
nonlinear processes.

In the literature for aggregation of linear process, see
\cite{Goncalvez1988, Granger, Linden, Lippi&Zaffaroni, Terence1},
usually introduces interaction between individual innovations
$\e^i_t$ considering that this can be decomposed into a common
innovation plus an independent innovation, i.e. $\e^i_t= u_t+
\xi^i_t$.  So, under the presence of common innovation, the
aggregation process keeps only the structure given by the common
component, whereas the aggregation of independent components are
washed out by aggregation, i.e. the convergence results obtained are
the same that considering only the common innovation component.

The immediate question is: it is possible to consider another kind
of interactive innovation that allows to obtain different results to
the cases of independent innovations and common innovation? We
answer this question affirmatively, in \cite{Dacunha&Fermin.L.Notes}
we have studied the aggregation of doubly stochastic gaussian
processes and we have approached the problem in a different way. We
introduce interaction between elementary processes ``living at $i$"
starting from interaction between innovations as
$\mathbb{E}[\varepsilon_t^i\varepsilon_t^j]=\chi(i-j)$, where $\chi$
is a  given  covariance function. Thus, the common innovation case
is given by $\chi(j)=1$, for all $j$, and the independent
innovations case by $\chi(j)=0$, for $j\neq 0$ and $\chi(0)=1$. The
procedure includes naturally the aforementioned case. Hence we
obtain some interesting qualitative behavior of the aggregation process

Some specific results already exist for aggregation of nonlinear
processes. Several authors have investigated the aggregation of $ARCH$ and
$GARCH$ processes, see \cite{Ding&Granger, Kazakevi, Leipus&Viano2002,
Nijman&Sentana, Zaffaroni}. All these works have been
developed considering exclusively common innovation or independent
innovations. As far as we know does not exist literature about the
aggregation of nonlinear processes with interactive innovations.

Our purpose is to extend the convergence results of
the aggregation procedure, given in \cite{Dacunha&Fermin.L.Notes}, for gaussian elementary processes with interactive innovations to case of nongaussian processes with
weakly dependent innovations, our results include elementary doubly
stochastic processes such as: nongaussian linear processes,
nonlinear Bernoulli shifts, Volterra processes, $ARCH$, $GARCH$,
$LARCH$ and bilinear models.

We consider a sequence of stationary doubly stochastic processes $Z=\{Z^i: \, i \in \mathbb{N}\}$  on $(\Omega, \mathcal{A}, P) \times (\Omega', \mathcal{A}', P')$ in $\mathbb{R}^{\Z}$ by
\begin{equation}\label{ec.6.1.1}
 Z^i= \left\{ Z^i_t \left( \y^i(\omega), \e^i(\omega') \right): \, t \in \Z \right\}.
\end{equation}
where $Y=\{\y^i: \, i \in \mathbb{Z}\}$  a sequence of random
variables defined on $(\Omega, \mathcal{A}, P)$, with distribution
$\mu$ on $\mathbb{R}^s$ and $\e = \{\varepsilon^i_t: \, i, \,t \in
\mathbb{Z} \}$ be a doubly indexed process defined on $(\Omega',
\mathcal{A}', P')$ satisfying the following assumption
\quad\\
\textbf{Assumption A1:}
\begin{enumerate}

  \item  \textit{$\Ep$ is an array of strong white noises, i.e. for each $i$, $\e^i=\{\e^i_t\}$ is an i.i.d sequence.} 

  \item \textit{$Y$ is an i.i.d sequence with distribution $\nu= \mu^{\bigotimes \N}$}.

  \item  \textit{$Y$ is independent of $\Ep$.}
\end{enumerate}

Here, the sequence $Y$ is considered as the random environment model and $\Ep= \{\e^i: \, i\in \Z\}$ is the sequence of innovations $\e^i= \{\e^i_t: \, t \in \Z\}$. For simplicity, throughout the work $Z^i_t:=Z_t(\y^i, \e^i)$, and we use $Z_t(\y^i, \e^i)$ only if we want to state explicitly the dependence on $\y^i$ and $\e^i$.

For every trajectory fixed $Y$, we define $X^N(Y)=\{X^N_t(Y):\, t \in \Z\}$ as the partial aggregation of the elementary processes $\{Z^i\}$ by
\begin{equation}\label{ec.6.1.2}
X^N_t(Y) = \frac{1}{B_N}\sum_{i=1}^N Z_t(\y^i, \e^i),
\end{equation}
where $B_N$ is an appropriate normalization sequence.

We study, for almost every fixed trajectory $Y$, the convergence of $X^N(Y)$ to some process $X$, called the aggregation process.

We introduce a new notion of weak dependence, for a sequence of doubly stochastic processes. This notion is based in the weak dependence definition given by Doukhan and Louhichi, see \cite{Doukhan&Louhichi}.

We give several models of doubly stochastic processes satisfying the new notion of weak dependence. To do that we consider innovations defined in such a form that, for each $t$, $\{\varepsilon^i_t: \, i \in \mathbb{Z}\}$ is a weakly dependent stationary sequence in the index $i$, see \cite{WDB-PDoukhan}. The notion of weak dependence for the innovations sequence $\{\e^i\}$ makes explicit the asymptotic independence between $(\e^{i_1}_{t_1}, \ldots, \e^{i_u}_{t_u})$ and $(\e^{j_1}_{t'_1}, \ldots, \e^{j_v}_{t'_v})$ when $i_1< \ldots i_u < i_u+ r \leq j_1< \ldots <j_v$ and $r$ tends to infinity.

Considering several types of weak dependence already used in the literature, we show that the weak dependence property of innovations is transferred to the sequence $\{Z^i\}$ of doubly stochastic elementary processes. In fact we will prove that $Z=\{Z^i\}$ satisfies the new notion of weak dependence. This result will be  prove for orthogonal expansion Volterra processes such as: linear processes, $ARCH$, $GARCH$, $LARCH$ and  bilinear models. We will also show this transference property for doubly stochastic uniform Lipchitz Bernoulli shift processes. For instance, Lipchitz functions of linear processes.

Let us suppose that elementary process $Z^i$ satisfies the following moment
condition
\begin{description}
\item[K$2_{\delta}$] \qquad $\mathbb{E}[|Z^i_t|^{2+\delta}]< \infty$, with $\delta>0$.
\end{description}
This conditions yields the existence of $Z^i$ in $L^2(\Omega \times
\Omega')$. We prove that under the weak dependence property, the
interaction $\chi$, defined by $\E[\e^i_t \e^j_t]= \chi(i-j)$, is
always a weak interaction in $\ell^1$. So $\chi$ has a finite limit in the Cesaro sense.

In general, we assume that the covariance function $\Gamma^N(Y)$ of partial aggregation process
$X^N(Y)$  satisfies
\begin{description}
\item[K4] \qquad $\Gamma^N(\tau,Y)$ converges $\nu-a.s.$ to $\Gamma(\tau)$, for all $\tau \in \mathbb{Z}$.
\end{description}

Thus, assuming that the doubly stochastic processes satisfy this new
notion of weak dependence and moreover the conditions K$2_{\delta}$
and K4, we prove a Central Limit Theorem $\nu-a.s.$ for the partial
aggregation $X^N(Y)$. This proof is done directly on the structure
of elementary doubly stochastic processes and not as in \cite{Dacunha&Fermin.aggreglinear},
for the aggregation of doubly stochastic linear process, which is
done on the structure of innovations.  This result includes some cases that have not been addressed in the literature. We can cite to illustrate the cases of nonlinear processes with weakly dependent innovations, such as general doubly stochastic orthogonal expansion Volterra processes and doubly stochastic uniform Lipchitz Bernoulli shift processes.

Finally, we show a new SLLN for the covariance function $\Gamma^N(Y)$ in the case of doubly stochastic orthogonal expansion Volterra processes, $DSV^*$ processes, with interactive innovations.

This paper is organized as follows. In Section~\ref{sect.6.1} we present some doubly stochastic Bernoulli shift processes and we give the necessary and sufficient condition to obtain their existence in $L^2(\Omega\times \Omega')$.  In Section~\ref{sect.6.2}
we introduce the notion of weak dependence for doubly stochastic
processes. In Section~\ref{sect.6.3} we present several examples
of weakly dependent innovations and we show that considering
different weakly dependent innovation models we can obtain weakly
dependent doubly stochastic models. In
Section~\ref{sect.6.4} we give the CLT $\nu-a.s.$ for
general sequence of weakly dependent doubly stochastic processes. In
Section~\ref{sect.6.5} we give a SLLN for $\Gamma^N$ in the
case of $DSV^*$ processes. Proofs are given in Section~\ref{sect.6.6}: we
prove the results relating to the transfer of weak dependence
property to $Z$, we develop the standard Linderberg method with
Bernstein's block and yield the CLT $\nu-a.s.$. Finally, we prove
the SLLN for $\Gamma^N(Y)$.

\section{Some doubly stochastic Bernoulli shift processes \label{sect.6.1}}

\hspace{0.4cm} In this section, we describe doubly stochastic versions of several
models used in statistics, econometrics and finance. We consider some stationary doubly stochastic elementary processes generated by a random parameter $\y$ defined on $(\Omega, \mathcal{A}, \Pro)$ with distribution $\mu$ on $\mathbb{R}^s$, and a strong white noise $\e= \{\e_t : \, t \in \Z \}$ on $(\Omega', \mathcal{A}', \Pro')$; i.e. doubly stochastic processes defined on $(\Omega, \mathcal{A}, \Pro) \times (\Omega', \mathcal{A}', \Pro')$ in $\mathbb{R}^{\Z}$ of the form \eqref{ec.6.1.1}

We introduce the doubly stochastic Bernoulli shifts, $DSBS$, a very broad class
of models that contains the major part of processes derived from a
stationary sequence. We define the $DSBS$ process in the following way, see
\cite{WDB-PDoukhan}.

\begin{definition}\label{def.4.2.1}
Let $\Hc : \R^s \times \R^{\Z} \rightarrow \R$ be a measurable
function. We define a $DSBS$ process $Z(\y, \e)$ with random
parameter $\y$ in $\R^s$ and innovation $\e= \{\e_k : \, k \in \Z \}$
by
\begin{equation}\label{ec.4.2.2}
Z_t(\y, \e)= \Hc \left(\y, \{\varepsilon_{t-k}\}_{ k\in \mathbb{Z}}
\right).
\end{equation}
\end{definition}

Since $\Hc$ is a function depending on an infinite number of
arguments, it is generally given in term of a series defined in
$L^2(\Omega \bigotimes \Omega')$. In order to define \eqref{ec.4.2.2} in a general setting, we denote for any subset $J \subset \Z$:
$$\Hc_{J}:= \Hc \left(\y,\{\e_j\}_{j \in J}\right) = \Hc\left(\y,
\{\e_j \1_{J}(j) \}_{j \in \Z}\right).$$

For finite subsets $J$ this expression is generally well defined and
simple to handle. In order to define such model in $L^2(\Omega
\bigotimes \Omega')$ we assume that the function $\Hc$ is such that
\begin{equation}\label{ec.4.2.3}
\E[|\Hc \left(\y, \e \right)|^2] < \infty.
\end{equation}
The $DSBS$ process $Z(\y,\, \e)= \{Z_t(\y,\, \e)\}$ is  stationary,
$\mu-a.s.$ Then, condition~\eqref{ec.4.2.3} implies that $Z(\y,\, \e)$ is on $L^2(\Omega \bigotimes \Omega')$, which implies its existence, $\mu-a.s.$, in $L^2(\Omega')$.

As we will see in the following, most of the usually used
stochastic processes can be represented as Bernoulli shifts.

As an example we consider doubly stochastic uniform Lipschitz Bernoulli shift ($DSULBS$) processes. We also study doubly stochastic Volterra processes, which correspond to random parameters Volterra processes and we detail the doubly stochastic bilinear models given afterwards their expression as a Volterra's expansion. Finally, we present several well known doubly stochastic bilinear models
such as $LARCH$, $GARCH$ and $ARCH$ processes with random coefficients. 

We will discuss the necessary and sufficient conditions for the existence $a.s.$ in  $L^2(\Omega\bigotimes\Omega')$ of these models. Those conditions will be very important for studying the convergence of the aggregation procedure.

\subsection{Doubly stochastic uniform Lipschitz Bernoulli shifts \label{sect.DSULBS}}

\hspace{0.4cm} Let $\Hc: \R^s \times \R^{\Z} \rightarrow \R$ be a function such
that if $\e, \e' \in \R^{\Z}$ coincide for all indexes but one, let
say $k_0 \in \Z$, then there exists $a_{k_0}(\y)$ $\mu-a.s.$ such that
\begin{equation}\label{ec.4.2.5}
\left| \mathcal{H}\left(\y, \{\e_{k}\}_{ k\in \mathbb{Z}} \right) -
\mathcal{H}\left(\y, \{\e'_{k}\}_{ k\in \mathbb{Z}} \right)\right|
\leq |a_{k_0}(\y)| |\e_{k_0} - \e'_{k_0}|.
\end{equation}
Thus,  there exists a sequence $a(\y)= \{a_k(\y): \, k \in \mathbb{Z}\}$, such that $\Hc$ satisfies $\mu-a.s.$,  for all $\e, \e' \in \R^{\Z}$, the following regularity condition
\begin{equation}\label{ec.4.2.6}
\left| \mathcal{H}\left(\y, \{\e_{k}\}_{ k\in \mathbb{Z}} \right) -
\mathcal{H}\left(\y, \{\e'_{k}\}_{ k\in \mathbb{Z}} \right) \right|
\leq \sum_{k\in \mathbb{Z}} |a_k(\y)| |\e_k - \e'_k|.
\end{equation}
We say that $\Hc$ is a uniform Lipschitz Bernoulli shift whit respect the coordinates $\e$.

Moreover supposing the following
\begin{equation}\label{ec.4.2.7}
\int \|a(\y)\|_{\ell_1} \mu(dy) < \infty,
\end{equation}
where $\|a(y)\|_{\ell_1}= \sum_k |a_k(y)|$. Condition~\eqref{ec.4.2.7} implies that $a(\y) \in \ell_1$, $\mu-a.s.$

Then, $Z(\y, \e)=\{Z_t(\y, \e): \, t \in \Z\}$ given by
$$Z_t= \mathcal{H}\left(\y^i, \{\varepsilon^i_{t-k}\}_{ k\in \mathbb{Z}} \right).$$
is defined as a doubly stochastic uniform Lipschitz Bernoulli shifts ($DSULBS$)
process.

Simple examples of this situation is Lipschitz function of linear processes. Another example
of this situation is the following stationary doubly stochastic process
$$Z_t(\y, \e)=\e_{t} \left(c_0(\y) + \sum_{k\neq 0} c_k(\y) \e_{t-k} \right),$$
where the innovation $\e_t$ is bounded and $c(\y)= \{c_k(\y): \, k \in \Z\}$ is such that $\E[\|c(\y)\|_{\ell_1}]< \infty$. In this case,
condition~\eqref{ec.4.2.5} also holds with $|a_k(\y)| = 2
\|\e_0\|^2_{\infty} |c_k(\y)|$.


\subsection{Doubly stochastic Volterra processes \label{sect.DSVP}}

\hspace{0.4cm} A doubly stochastic Volterra (DSV) process is defined through a convergent Volterra expansion with random coefficients
\begin{equation}\label{ec.4.3.1}
Z_t(\y, \e)= \sum_{k=0}^{\infty} \sum_{l_1,\ldots,l_k \in \Z^k}
c_{k:l_1,\ldots,l_k}(\y) \varepsilon_{t-l_1}\cdots\varepsilon_{t-l_k},
\end{equation}
where $\e$ is a sequence of innovation and $c(\y) = \{c_{k:l_1,\ldots,l_k}(\y): \, (l_1,\ldots,l_k) \in \mathbb{Z}^k, \,  k \in \mathbb{Z}\}$ is a sequence of random
coefficients.

We assume that $\e$ is a strong white noise such that, for all $k \in \N$, $\E[|\e_t|^{2k}]<\infty$. Let us note  $\ell_{\e}$ the space of sequence $c=\{c_{k:l_1,\ldots,l_k}: \, (l_1,\ldots,l_k) \in \Z^k, \, k \in \Z\}$ such that
\begin{equation}\label{ec.4.3.2}
\|c\|^2_{\ell_{\e}}=\sum_{k=0}^{\infty}
\mathbb{E}[|\varepsilon_t|^{2k}]\sum_{l_1,\ldots,l_k}
|c_{k:l_1,\ldots,l_k}|^2 < \infty.
\end{equation}

If $c(\y)$ satisfies condition~\eqref{ec.4.3.2} $\mu-a.s.$, this entails the $\mu-a.s.$ convergence in $L^2(\Omega')$ of the series \eqref{ec.4.3.1}. Since $\e$ is a strong white noise  a sufficient condition for the convergence in $L^2(\Omega \times \Omega')$ of the series \eqref{ec.4.3.1} is
\begin{description}
\item[C2] \qquad $\E[\|c(\y)\|^2_{\ell_{\e}}]< \infty$.
\end{description}

In order to obtain condition~C2 it is necessary the existence of all moments of
$\e$, hence is sufficient the existence of Laplace transform of distribution $\mathbb{P}_{\e}$. This holds, for instance, if $\e$ is a sub-gaussian process.

\begin{remark}
We note that the $DSV$ process correspond to a $DSBS$ process in the case of chaotic expansion associated with the discrete chaos generated by the innovation $\e$, \cite{Schreiber, Neveu}. We can consider that, for each $\y$ fixed, $\Hc(\y,\cdot)$ is expanded, in $L^2$ sense, into this chaos as
$$\Hc(\y, \e)= \sum_{k=0}^{\infty} \Hc^{(k)}(\y, \e),$$
where $\Hc^{(k)}(\y, \e)$ denotes the $k$-th order chaotic
contribution, $\Hc^{(0)}= c_{0}(\y)$ and
$$\Hc^{(k)}(\y, \e)= \sum_{l_1, \ldots, l_k \in \Z^k} c_{k: l_1, \ldots, l_k}(\y) \e_{l_1}\ldots \e_{l_k}.$$
The process $\Hc^{(k)}(\y, \{\e_{t-j}\}_{j\in \Z})$ is called the
$DSV$ process of order $k$.
\end{remark}

The most simple $DSV$ process are the linear processes with random
coefficients, defined by
\begin{equation} \label{ec.4.3.3}
Z_{t}(\y,\e)=\sum_{k\in \mathbb{Z}} c_k(\y)\varepsilon_{t-k}.
\end{equation}
These linear processes are $DSV$ processes of order one.

An interesting particular case of Volterra processes is obtained when the chaotic expansion is given by
\begin{equation}\label{ec.4.3.4}
Z_t(\y, \e)=  \sum_{k=0}^{\infty} \sum_{l_1<\ldots<l_k} c_{k:l_1,\ldots,l_k}(\y) \e_{t-l_1}\cdots\e_{t-l_k}.
\end{equation}

The corresponding $k$-th order Volterra processes are pairwise orthogonal. In the following, we will referee this type of doubly stochastic Volterra processes with orthogonal expansion by $DSV^*$ processes. Since $Z_t(\y, \e)$ is a sum of orthogonal terms, if the innovations satisfies $\E[|\e_t|^2]=1$, the series \eqref{ec.4.3.4} is $\mu-a.s.$ convergent in $L^2(\Omega')$ if and only if
\begin{equation}\label{ec.4.3.5}
\|c(\y)\|^2_2= \sum_{k=0}^{\infty} \sum_{l_1<\ldots<l_k}
|c_{k:l_1,\ldots,l_k}(\y)|^2 < \infty,  \qquad \mu-a.s.
\end{equation}
In this case, we put $\ell_{\epsilon}= \ell_2$, the space of sequence $c$
such that $\|c\|_{\ell_{\e}}= \|c\|_{2} < \infty$. Then, we have
that $Z(\y, \e)$ belongs to $L^2(\Omega \bigotimes \Omega')$ if and only if  condition~C2 is satisfied.

The difference with the general case is that the indexes in the
product are all different. So that, in the general case, we do not
get the orthogonality property, which is a disadvantage, see \cite{Doukhan.LRM}.

Nevertheless, the i.i.d sequence $\e=\{\e_t: \, t \in \Z\}$ can be replaced by a vector-valued i.i.d sequence $\left\{\left(\e^{(k,1)}_t, \ldots, \e^{(k,k)}_t \right): \, t \in \Z \right\}$, such that $\left\{\e^{(k,l)}: l\leq k \right\}$ are mutually orthogonal. This is obtained by replacing each power of an innovation variable with its decomposition on the Appell polynomials of the distribution of $\e_0$, see e.g. \cite{Doukhan.LRM}.

Then we can rewrite the process, under condition~C2, as a sum of orthogonal terms given by

\begin{equation}\label{ec.4.3.6}
Z_t(\y, \e)=  \sum_{k=0}^{\infty} \sum_{l_1<\ldots<l_k}
c_{k:l_1,\ldots,l_k}(\y) \e^{i,(k,1)}_{t-l_1}\cdots
\e^{i,(k,k)}_{t-l_k}.
\end{equation}

In the following section, we present some examples of $DSV^*$ processes. These models are
a natural extension of models that are well known in the literature and have many sorts of applications.


\subsection{Doubly stochastic bilinear models \label{sect.DSBM}}

\hspace{0.4cm} A vast literature is devoted in order to study the conditionally
heteroscedastic models. The heteroscedasticity property is characterized by the fact
that the conditional variance $var(Z_t | I_{t-1})$, given an information set $I_{t-1}$, is not a constant.

The bilinear models allows unify the treatment of heteroscedastic models. These models are defined by the relation
\begin{equation}\label{ec.4.4.1}
Z_t(\y, \e)= \left(a_0(\y) + \sum_{k=1}^{\infty} a_k(\y) Z_{t-k}\right)
+ \left(b_0(\y) + \sum_{k=1}^{\infty} b_k(\y) Z_{t-k}\right)\e_t.
\end{equation}
where $\e=\{\e_t\}$ is a sequence of  i.i.d. centered random variables such that $\E[|\e_t|^2]=1$, and $a(\y)= \{a_k(\y)\}$, $b(\y)=\{b_k(\y)\}$  are real random coefficients, not necessary nonnegative.

This models appears naturally when studying the class of processes with the property that the conditional mean $m_t(\y)=\E^\y[Z_t| Z_k,\, k < t]$  is a linear combination of $Z_k$, $k<t$, and the conditional variance $\sigma_t^2(\y)= var(Z_t|Z_k,\, k < t)$ is the square of a linear combinations of $Z_k$, for $k<t$, as it is in the case of \eqref{ec.4.4.1}
$$\sigma^2_t(\y)= \left(b_0(\y) + \sum_{k=1}^{\infty} b_k(\y) Z_{t-k}\right)^2 \quad \text{and} \quad m_t(\y)= \left(a_0(\y) + \sum_{k=1}^{\infty} a_k(\y) Z_{t-k}\right).$$
Let us take
\begin{equation*}
\begin{array}{llllllcll}
a(s, \y) & = & \sum_{k=1}^{\infty} a_k(\y) s^k, & \quad & g(s,\y)
& = & (1-b(s,\y))^{-1}& = & \sum_{k=0}^{\infty} g_k(\y) s^k , \\
b(s, \y) & = & \sum_{k=1}^{\infty} b_k(\y) s^k, & \quad & h(s,\y) & = &
a(s,\y)(1-b(s,\y))^{-1} & = & \sum_{k=0}^{\infty} h_k(\y) s^k .
\end{array}
\end{equation*}
If we suppose that $H(\y)= \sum_{l=1}^{\infty} h^2_l(\y)< 1$ $\mu-a.s.$,
then there is, $\mu-a.s.$, a unique second order stationary solution
given by
\begin{equation}\label{ec.4.4.2}
Z_t( \y, \e) = b_0(\y) \sum_{k=1}^{\infty} \sum_{0\leq l_1< \ldots<
l_k} g_{l_1}(\y)h_{l_2-l_1}(\y)\ldots h_{l_{k}-l_{k-1}}(\y)
 \e_{t-l_1}\ldots \e_{t-l_k}.
\end{equation}
Then, $Z(\y, \e)$ is a $DSV^*$ process with random coefficients
$$c_0(\y)=b_0(\y) \quad \text{and } \quad c_{k: l_1, \ldots l_k}(\y) =  g_{l_1}(\y)h_{l_2-l_1}(\y)\ldots h_{l_{k}-l_{k-1}}(\y).$$

We present the necessary and sufficient conditions for the
existence in $L^2$ of process $Z(\y, \e)$ defined by
\eqref{ec.4.4.1}, see \cite{Giraitis&Surgailis}.

Let $G(\y)= \sum_{l=0}^{\infty} g^2_l(\y)$. If $b_0(\y)< \infty$ and
$H(\y)<1$ $\mu-a.s.$ we obtain
\begin{equation}\label{ec.4.4.2bis}
\|c(\y)\|_2^2 = b^2_0(\y) \sum_{k=1}^{\infty} \sum_{0\leq l_1<
\ldots
< l_k } g^2_{l_1}(\y)h^2_{l_2 - l_{1}}(\y) \ldots h^2_{l_k-l_{k-1}}(\y) =\frac{b_0^2(\y)G(\y)}{1-H(\y)}.
\end{equation}
Then, from \eqref{ec.4.4.2bis}, we obtain that the following conditions
\begin{eqnarray}\label{ec.4.4.3}
H(\y) < 1  \quad \mu-a.s. \quad \text{ and } \quad \E\left[\frac{b_0^2(\y)G(\y)}{(1-H(\y))}\right] < \infty.
\end{eqnarray}
are necessary and sufficient  for the existence of a stationary
$L^2$-solution to \eqref{ec.4.4.1}. This solution $Z(\y,\e)$ is
given by \eqref{ec.4.4.2}.

Formally the classes $AR$, $ARMA$, $ARCH$, $GARCH$, $LARCH$ with random coefficients all belong to the class of doubly stochastic bilinear models. In the follows we details the necessary and sufficient condition for the existence of $L^2$-solutions for some of those classes.

\subsubsection{Doubly stochastic $LARCH(\infty)$ processes}

\hspace{0.4cm} We consider the doubly stochastic $LARCH(\infty)$ models given by
\begin{equation}\label{ec.4.4.4}
Z_t= \left(b_0(\y) + \sum_{k=1}^{\infty} b_k(\y) Z_{t-k}\right)\e_t.
\end{equation}
with real value random coefficients, not necessary nonnegatives, and $\e=\{\e_t\}$ a white noise. The $LARCH$ model with deterministic coefficients has been introduced by Robinson, see \cite{RobinsonARCH}. For a vectorial version of this model, see \cite{Doukhan.et.al2006}.

Let us denote $B_2(\y)= \sum_{k=1}^{\infty} b_k^2(\y)$ and  $ \E[|\e_t|^2]=1$. Following the same way that for bilinear models we obtain that conditions
\begin{equation}\label{ec.4.4.5}
B_2(\y) < 1  \quad  \mu-a.s. \quad \text{and }\quad
\E\left[\frac{b_0^2(\y)}{1-B_2(\y)}\right] < \infty
\end{equation}
are necessary  and sufficient for the existence of a stationary
$L^2$-solution to \eqref{ec.4.4.4}.

The stationary solution of equation~\eqref{ec.4.4.4} has a
orthogonal Volterra expansion given by equation
\begin{equation}\label{ec.4.4.6}
Z_t(\y, \e)= \sum_{k=0}^{\infty} \sum_{0<l_1<, \ldots,< l_k} b_0(\y)
b_{l_1}(\y)b_{l_1-l_2}(\y)\ldots b_{l_k-l_{k-1}}(\y) \e_t
\e_{t-l_1}\ldots \e_{t-l_k}.
\end{equation}

\subsubsection{Doubly stochastic $ARCH(\infty)$ processes}

\hspace{0.4cm} Here we consider the random coefficients nonnegative $ARCH$ model
defined as follow. A process $Z(\y,\e)=\{Z_t(\y,\e) :\, t \in \Z\}$ is
said to satisfy random coefficient $ARCH(\infty)$ equations if there
exist a nonnegative i.i.d. innovation sequence $\e= \{\e_t: \, t\in
\Z \}$ and nonnegative random variables $\{b_k(\y)\}$ independent of
$\e$ such that
\begin{equation}\label{ec.4.4.7}
Z_t=\left(b_0(\y) + \sum_{k=1}^{\infty} b_k(\y) Z_{t-k}\right)\e_t
\quad \text{a.s.}
\end{equation}
Note that the random coefficient $ARCH$ process given by
equation~\eqref{ec.4.4.7} is a nonergodic process.

Classically, $ARCH$ process we mean the model where the returns
$r_t$ admit a representation of the form
\begin{equation}\label{ec.4.4.8}
r_t= \sigma_t \xi_t, \qquad \sigma^2_t = b_0 +\sum_{k=1}^{\infty}
b_k r_{t-k}^2.
\end{equation}
where $\xi=\{\xi_t\}$ is a sequence of i.i.d random
variables with zero mean and finite variance, and $\sigma^2_t$ is a
linear combination of the squares of past returns. The $GARCH(p,q)$
model is defined by
\begin{equation}\label{ec.4.4.9}
r_t= \sigma_t \xi_t, \qquad \sigma^2_t = \alpha_0 +
\sum_{k=1}^{p}\beta_k \sigma^2_{t-k} + \sum_{k=1}^{q} \alpha_k
r_{t-k}^2.
\end{equation}
Under some restrictions on the polynomials $\alpha(z)= \sum_{k=1}^{q}
\alpha_k z^k$ and $\beta(z)=\sum_{k=1}^{p}\beta_k z^k$, the model
can be rewritten in the form \eqref{ec.4.4.8}.

Then, denoting $Z_t= r_t^2$ and  $\e_t= \xi_t^2$, we represent this
model in the $ARCH$ form \eqref{ec.4.4.7}, see \cite{RobinsonARCH}.

In the case of nonrandom coefficients, those models are introduced by
\cite{RobinsonARCH} and subsequently studied in \cite{Giraitis, Kazakevi&Leipus, Kokoszka&Leipus}. For the case of random coefficients see \cite{Kazakevi&Leipus, Leipus&Viano2000}.

Let $\lambda_1= \E[\e_t]$, $\lambda_2= \E[\e_t^2]$ and $B(\y)=
\sum_{k=1}^{\infty}b_k(\y)$. The recursion relation \eqref{ec.4.4.7}
yields the following Volterra series expansion of $Z(\y,\e)$:
\begin{equation}\label{ec.4.4.10}
Z_t(\y, \e)= \sum_{k=0}^{\infty} \sum_{0< l_1<\ldots <\l_k} b_0(\y)
b_{l_1}(\y)\ldots b_{l_k -l_{k-1}}(\y)  \e_t \e_{t-l_1} \ldots
\e_{t-l_k}.
\end{equation}

In order to obtain the necessary and sufficient conditions for the existence in $L^2$ of process $Z(\y, \e)$ defined by \eqref{ec.4.4.10}, one needs
to study orthogonal Volterra representation of $Z(\y, \e)$. This
orthogonal representation is obtained by replacing the $\e_t$'s by
$\e_t=\lambda_1(\kappa \te_t +1)$, where $\kappa^2=
(\lambda_2-\lambda_1^2)/\lambda_1^2$ and $\te_t$ have zero mean and
unit variance. We can rewrite \eqref{ec.4.4.7} of the bilinear form, see \cite{Giraitis&Surgailis},
\begin{eqnarray*}
Z_t(\y,\e)&=& \left(\lambda_1 b_0(\y) + \sum_{k=1}^{\infty} \lambda_1b_k(\y) Z_{t-k}\right) +\left(\kappa\lambda_1 b_0(\y) + \sum_{k=1}^{\infty} \kappa\lambda_1 b_k(\y) Z_{t-k}\right)\te_t.
\end{eqnarray*}

Then, similarly to case of bilinear models, we obtain that $Z(\y, \e)$ the following $DSV^*$ representation where
\begin{eqnarray*}\label{ec.4.4.15}
Z_t(\y,\e)&=& \frac{\lambda_1 b_0(\y) }{1-\lambda_1B(\y)}\left( 1 +
\sum_{k=1}^{\infty} \sum_{0\leq l_1< \ldots < l_{k} }
h_{l_1}(\y)h_{l_2-l_{1}}\ldots h_{l_k-l_{k-1}}(\y)\te_{t-l_1}\ldots
\te_{t-l_k}\right),
\end{eqnarray*}
with $g_0=1$,
\begin{equation*}
\sum_{k=0}^{\infty} g_k(\y) s^k = \left(1-\lambda_1\sum_{k=1}^{\infty} b_k(\y)s^k \right)^{-1},
\end{equation*}
and $h_k(\y)= \kappa g_k(\y)$, for $k\geq1$. More explicitly,
\begin{equation*}
g_l(\y)= \sum_{k=1}^l\lambda_1^k \sum_{0< j_1< \ldots <
j_{k-1}< l } b_{j_1}(\y)b_{j_2-j_1}\ldots b_{l -j_{k-1}}(\y).
\end{equation*}

Let $H(\y)= \sum_{l=1}^{\infty} h^2_l(\y)$. Thus, we obtain that the following conditions
\begin{eqnarray}\label{ec.4.4.16}
\lambda_1 B(\y) < 1,   \quad  H(\y) < 1  \quad \mu-a.s. \\
\E\left[\frac{b_0^2(\y)}{(1-\lambda_1B(\y))^2 (1-H(\y))}\right] <
\infty.
\end{eqnarray}
are necessary and sufficient  for the existence of a stationary
$L^2$-solution to \eqref{ec.4.4.7}.

Now, we present an examples of $ARCH$ models.

\begin{example}[Doubly stochastic $GARCH(1,1)$ processes]

Let us consider the random parameters $GARCH(1,1)$ model given by
\begin{equation*}
r_t(\y)= \sigma_t(\y) \xi_t, \qquad \sigma^2_t(\y) = \alpha_0(\y) +
\alpha(\y) \sigma^2_{t-1}(\y) + \beta(\y) r_{t-1}^2(\y),
\end{equation*}
where $\alpha_0(\y)$, $\alpha(\y)$ and $\beta(\y)$ are nonnegative
random variables independent of $\{\xi_t\}$. We note that the
corresponding $ARCH(\infty)$ equation \eqref{ec.4.4.8} is obtained
taking the random parameters as
$$b_0(\y) = \frac{\alpha_0(\y)}{1-\beta(\y)} \quad \text{ and }  \quad  b_k(\y)= \alpha(\y) \beta^{k-1}(\y).$$
Let $\varrho(\y)= \lambda_1\alpha(\y)+\beta(\y)$. In this case $h_k(\y)=
\lambda_1\kappa\alpha(\y)\varrho(\y)^{k-1}$, then
$$B(\y)=\frac{\alpha(\y)}{1-\beta(\y)} \quad \text{ and }  \quad H(\y)= \frac{\lambda_1^2\kappa^2 \alpha^2(\y)}{1-\varrho^2(\y)}.$$
So, $Z(\y,\e)$ exist in  $L^2$ if and only if
\begin{eqnarray}\label{ec.4.4.17}
\varrho^2(\y)+ \lambda_1^2\kappa^2\alpha^2(\y)< 1 \qquad \mu-a.s. \\
\E\left[\frac{\alpha_0^2(\y)}{(1-\varrho)^2 (1-\varrho^2(\y)-
\lambda_1^2\kappa^2\alpha(\y)^2)}\right] < \infty.
\end{eqnarray}

\end{example}

\begin{example}[Doubly stochastic $ARCH(1)$ processes]

We consider the random parameters $ARCH(1)$ model given by
\begin{equation*}
r_t(\y)= \kappa_t(\y) \xi_t, \qquad \kappa^2_t(\y) = \alpha_0(\y) +
\alpha(\y) \kappa^2_{t-1}(\y),
\end{equation*}
where $\alpha_0(\y)$ and $\alpha(\y)$ are nonnegative random variables
independent of $\{\xi_t\}$. We note that the corresponding $ARCH(1)$
process is a particular case of the $GARCH(1,1)$ process, when we
take the parameter $\beta(\y)=0$. Thus, $b_0(\y) = \alpha_0(\y)$,
$b_1(\y) = \alpha(\y)$ and $b_k(\y)=0$ for $k>1$.
In this case $h_k(\y)= \kappa (\lambda_1\alpha(\y))^{k-1}$. Then
$Z(\y,\e)$ exists in $L^2$ if and only if
\begin{eqnarray}\label{ec.4.4.18}
\sqrt{\lambda_2}\alpha(\y)< 1 \quad \mu-a.s. \quad \\
\E\left[\frac{\alpha_0^2(\y)}{(1-\lambda_1\alpha(\y))^2
(1-\lambda_2\alpha^2(\y))}\right] < \infty.
\end{eqnarray}

\end{example}


\section{Weak dependence for doubly stochastic processes \label{sect.6.2}}

\hspace{0.4cm} In this section we introduce a notion of weak dependence for
sequence of doubly stochastic processes following the definition of
weak dependence given by Doukhan and Louhichi, see \cite{Doukhan&Louhichi}.

Let $\Delta^{(k)}$ be the set of bounded Lipschitz functions $f$
defined on $\mathbb{R}^k$ such that for all $(z_1,\ldots,z_k)$ and
$(z^{'}_1,\ldots,z^{'}_k)$ in $\mathbb{R}^k$
$$|f(z_1,\ldots,z_k) - f(z^{'}_1,\ldots,z^{'}_k)|\leq
Lip(f)\sum_{m=1}^k |z_m - z^{'}_m|.$$
Above we denote the Lipschitz constant of $f$ by $Lip(f)$. Let
$\Delta^{(k)}_1$ be the set of function $f$ in $\Delta^{(k)}$ such
that $\|f\|_{\infty}\leq 1$.

Let $\{Z_i\}$ be a real-valued stationary process.

\begin{definition}[Doukhan and Louhichi] \label{def.6.2.1}
The process $\{ Z_i : \, i \in \mathbb{Z}\}$ is $(\epsilon,
\psi)$-weakly dependent if there exist a sequence ${\epsilon(r)}$
decreasing to zero at infinity and a function $\psi$ from
$\mathbb{N}^2\times (\R^{+})^2$ to $\mathbb{R}^{+}$ such that
$$\left|cov \left(f( Z_{i_1},..., Z_{i_u}),
\, g( Z_{j_1},...,Z_{j_v})\right)\right| \leq
\psi(u,v,Lip(f),Lip(g))\epsilon(r),$$
for any $r\geq 0$ and any $(u+v)$-tuples such that $i_1<\ldots< i_u
\leq i_u + r< j_1 <\ldots<j_v$, where $(f,g) \in \Delta^{(u)}_1
\times \Delta^{(v)}_1$.
\end{definition}

We introduce next $\epsilon$  the dependence coefficient.

\begin{remark}\label{rem.6.1.1}
The $\epsilon$-coefficients depend on the form of function $\psi$.
Let $\mathcal{J}(u,v,r)$ the set of $(u+v)$-tuples such that
$i_1<\ldots< i_u \leq i_u + r< j_1 <\ldots< j_v$. Given a function
$\psi$  from $\mathbb{N}^2\times (\R^{+})^2$ to $\mathbb{R}^{+}$,
the $\epsilon$-coefficient associated to $\psi$ is defined by
$$
\epsilon(r)= \sup_{u,v} \sup_{\mathcal{J}(u,v,r)} \sup_{(f,g) \in
\Delta^{(u)}_1 \times \Delta^{(v)}_1} \frac{\left|cov \left(f(
Z_{i_1},..., Z_{i_u}), \, g(
Z_{j_1},...,Z_{j_v})\right)\right|}{\psi(u,v,Lip(f),Lip(g))}
$$
\end{remark}

Specific functions $\psi$ yield notion of weak dependence
appropriated to describe different models. In the following, we will
consider different types of $\psi$ which are used in the current
bibliography of weak dependence, see \cite{WDB-PDoukhan,
Doukhan&Louhichi, Doukhan&Wintenberger}.

\begin{itemize}
\item $\lambda$-weak dependence: the $\lambda$-coefficient corresponds to
$\psi(u, v, a, b)= a u  + b v +  a b u v$, in this case we simply
denote $\epsilon(r) = \lambda(r)$.

\item $\eta$-weak dependence: the $\eta$-coefficient corresponds to
$\psi(u, v, a, b)= a u + b v $, we denote $\epsilon(r)
= \eta(r)$.

\item $\theta$-weak dependence: $\theta$-coefficient corresponds to $\psi(u, v, a,
b)= b v$, in this case we write $\epsilon(r) = \theta(r)$. This is
the causal counterpart of $\eta$-coefficients.

\item $\kappa$-weak dependence: in this case the  $\kappa$-coefficient corresponds to
$\psi(u, v, a, b)= a b u v$, and we write $\epsilon(r) = \kappa(r)$.

\item  $\kappa'$-weak dependence: the $\kappa'$-coefficient  correspond to
$\psi(u, v, a, b)=  a bv$, in this case we denote $\epsilon(r) =
\kappa'(r)$. This is the causal counterpart of $\kappa$-coefficient.

\end{itemize}

We now extend the notion of weak dependence to a sequence of doubly
stochastic processes. We consider $Z= \{Z^i \, i \in \mathbb{Z} \}$
a stationary sequence of stochastic processes, $Z^i= \{Z^i_t: \,
t\in \Z\}$.

\begin{definition}\label{def.6.2.2}
We say that $Z=\{Z^i: \, i \in \mathbb{Z} \}$
is $(\epsilon, \psi, Y)$-weakly dependent if, conditionally to $Y$
and for almost all trajectory $Y$, there exist a sequence
${\epsilon(r)}$ decreasing to zero at infinity, a function $\psi$
from $\mathbb{\mathbb{R}}^2\times (\R^{+})^2$ to $\mathbb{R}^{+}$
and a positive random variable $V(\y)$ such that,
$$\left|cov \left(f(Z^{i_1}_{t_1},...,Z^{i_u}_{t_u}), \, g( Z^{j_1}_{t'_1},...,Z^{j_v}_{t'_v}) \right) \right| \leq
\psi(d_{\mathbf{i_u}}(Y)\, , d_{\mathbf{j}_v}(Y), \, Lip(f), \,
Lip(g))\, \epsilon(r),$$
for all $(t_1, \ldots, t_u) \in \Z^u$ and $(t'_1, \ldots, t'_v) \in
\Z^v$, for any $r\geq 0$ and for any $(u+v)$-tuples such that
$i_1<\ldots< i_u \leq i_u + r< j_1 <\ldots<j_v$,  where $(f,g) \in
\Delta^{(u)}_1 \times \Delta^{(v)}$, $\mathbf{i}_u=(i_1,\ldots,
i_u)$, $\mathbf{j_v}=(j_1,\ldots, j_v)$ and
$$d_{\mathbf{i}_u} = \sum_{m=1}^u V(\y^{i_m}), \qquad d_{\mathbf{j}_v}= \sum_{m=1}^v V(\y^{j_m})$$
and
$\mathbb{E}^{\y^i}[|Z^i_{t}|^2] \leq V^2(\y^i)< \infty$,
$\mu.a.s.$
\end{definition}

\begin{remark}
When measure $\mu$ is degenerate, i.e when $\y^i= \y$ for all $i$,
then if the stationary process $Z= \{Z_t(\varepsilon^i): \, i, t \in
\mathbb{Z}\}$ satisfies Definition~\ref{def.6.2.2} we simply say
that $Z$ is  $\epsilon$-weakly dependent.
\end{remark}

We introduce the following condition:
\begin{description}
\item[K5.] \qquad $\mathbb{E}[V(\y)] < \infty$.
\end{description}
This condition implies that $V(\y)< \infty$, $\mu.a.s.$

\begin{remark}
In definition~\ref{def.6.2.2} the random variable $V^2(\y)$
establish a control of the $\y$-conditional variance of $Z_t(\y, \e)$.
We will see that in some case $\mathbb{E}^{\y}[|Z_{t}(\y
\e)|^2] = V^2(\y)$, for instance for $DSV^*$ processes.
\end{remark}


\section{Transference of the weak dependence property to doubly stochastic models  \label{sect.6.3}}

\hspace{0.4cm}  We show that for different weakly dependent
innovation  models  we can obtain weakly dependent doubly stochastic
models.

Let us consider doubly stochastic Bernoulli shift processes
defined, as in Section~\ref{sect.6.1}, by
\begin{equation}\label{ec.6.1.2}
Z_t(\y^i, \e)= \Hc \left(\y^i, \{\varepsilon^i_{t-k}\}_{ k\in
\mathbb{Z}} \right).
\end{equation}
where $\Hc : \R^s \times \R^{\Z} \rightarrow \R$ be a measurable
function in $L^2(\Omega \bigotimes \Omega')$ and we suppose that,
for each $t$, $\{\varepsilon^i_t: i \in \mathbb{Z}\}$ is weakly
dependent. We show that such property of weak dependence for
the innovations is transferred to the doubly stochastic processes.
We present some examples, already known in the literature, of weakly
dependent innovations.


\subsection{Examples of weakly dependent innovations \label{sect.6.3.1}}

\hspace{0.4cm} In the following, we give examples of weakly dependent innovations,
for more details to respect of theses innovation models see
\cite{WDB-PDoukhan, Doukhan&Wintenberger}.
\begin{remark}\label{rem.6.4.5}
In this work, we denote by $\mathcal{H}_{\varepsilon}$ to Bernoulli
shift innovation and simply by $\mathcal{H}$ to doubly stochastic
Bernoulli shift.
\end{remark}
%


\subsubsection{Bernoulli shift innovations}

\hspace{0.4cm} Let us take $\varepsilon^i_t=
\mathcal{H}_{\varepsilon}\left(\{\xi_{t}^{i-l}\}_{l\in
\mathbb{Z}}\right)$ a Bernoulli shift with input $\{\xi_{t}^{l}\}$. We
assume that, for all $t$ fixed, the shift $\mathcal{H}_{\varepsilon}$ satisfies the
following condition
\begin{equation}\label{ec.6.3.4}
\sum_{n=1}^{\infty} w_{n} < \infty.
\end{equation}
where $$w_{n}= \mathbb{E}\left[\left|\mathcal{H}_{\varepsilon}\left(
\{\xi_{t}^{i-l}\}_{|l|\leq n }\right) -
\mathcal{H}_{\varepsilon}\left(\{\xi_{t}^{i-l}\}_{|l|< n }\right)
\right|^2\right]^{\frac{1}{2}}.$$
This condition indeed proves that the sequence
$\{\mathcal{H}_{\varepsilon}\left( \{\xi_{t}^{i}\}_{|l|\leq n
}\right)\}$ has the Cauchy property in the space $L^2(\Omega')$ and
so its convergent. Let $\delta_r= \sum_{n \geq r} w_{n}$ then
$\{\delta_r\}_{r \in \mathbb{N}}$ converges to zero as $r
\rightarrow \infty$.

Doukhan et al. prove the following results, for each $t$ fixed, see Lemma 3.1, Lemma 3.2
and Lemma 3.3 in \cite{WDB-PDoukhan}.

\begin{theorem}[Doukhan et al.]\label{theo.6.3.2}\quad\\\vspace{-0.3cm}
\begin{itemize}
\item \textbf{Noncausal shift innovations with independent inputs}\\
If $\{\xi_{t}^{l}\}$ a sequence of i.i.d. random variables then the
innovations $\{ \varepsilon^i_t : \, i \in \mathbb{Z}\}$ are
$\eta$-weakly dependent with $\eta(r) \leq 2 \delta_{\lfloor r/2
\rfloor}$.

\item \textbf{Noncausal shift innovations with dependent inputs}\\
If $\{\xi_{t}^{l}: \, l \in \mathbb{Z}\}$ is $\eta_{\xi}$-weakly
dependence then the innovations $\{ \varepsilon^i_t : \, i \in
\mathbb{Z}\}$ are $\eta$-weakly dependents.\\
If $\{\xi_{t}^{l}\}$ is $\lambda_{\xi}$-weakly dependent then
$\{\varepsilon^i_t : \, i \in \mathbb{Z}\}$  are $\lambda$-weakly
dependents. Nevertheless, if the input sequences are $\kappa$-weakly
dependents then the innovations $\{\varepsilon^i_t : \, i \in
\mathbb{Z}\}$ are neither $\kappa$ nor $\eta$-weakly dependents.

\item  \textbf{Causal shift innovations with independent inputs}\\
Let us take $\varepsilon^i_t=
\mathcal{H}_{\varepsilon}\left(\{\xi_{t}^{i-l}\}_{l\in
\mathbb{N}}\right)$ a causal Bernoulli shift with input
$\{\xi_{t}^{l}\}$ a sequence of i.i.d. random variables. Then, the
innovations $\{\varepsilon^i_t: \, i \in \mathbb{Z}\}$ are
$\theta$-weakly dependent with $\theta(r)\leq 2\delta_r$ and
$\delta_r= \sum_{n \geq r} w_{n}$.
\end{itemize}
\end{theorem}

\begin{example}[Linear innovations]\label{ex.Linnov}
If $\varepsilon^i_t= \sum_{l \in \mathbb{Z}} \beta_l \xi_{t}^{i-l}$,
$\|\beta\|_2^2= \sum_l \beta_l^2 < \infty$ with $\{\xi_{t}^{l}\}$ a
sequence of i.i.d. random variables, then $\varepsilon^i_t$ is $\eta$-weakly dependent with
$\eta(r) \leq \left(\sum_{|l|\geq r}\beta_l^2\right)^{\frac{1}{2}}$.
\end{example}

\begin{example}[Volterra innovations]\label{ex.Vinnov}
If we suppose that $\varepsilon_t^i$ is defined by the following
Volterra expansion
$$\varepsilon^i_t= \sum_{k=0}^{\infty} \sum_{l_1<,\ldots,<l_k}
\beta_{k;l_1,\ldots,l_k} \xi_{t}^{i-l_1}\cdots\xi_{t}^{i-l_k} $$
where $\{\beta_{k;l_1,\ldots,l_k}: \, (l_1,\ldots,l_k) \in
\mathbb{Z}^k\}$ is a sequence of real numbers and $\{\xi_{t}^{l}\}$ is a
sequence of i.i.d. random variables. This expression converges in
$L^2(\Omega')$ provided that
$\mathbb{E}\left[|\xi_{t}^{i}|^2\right]<\infty$ and
$$\|\beta\|_2^2= \sum_{k=0}^{\infty} \sum_{l_1<,\ldots,<l_k}
|\beta_{k;l_1,\ldots,l_k}|^2 < \infty.$$
In this case, we can verify that $\varepsilon^i_t$ is $\eta$-weakly dependent.
$$\eta(r)^2\leq \|\beta\|_2^2 - \sum_{k=1}^{2r-1} \sum_{-r<l_1<,\ldots,<l_k<r}
|\beta_{k;l_1,\ldots,l_k}|^2.$$
Condition $\|\beta\|_2< \infty$ implies that $\eta(r)$
converges to zero as $r \rightarrow \infty$.

In the general case of Volterra innovation
$$\varepsilon^i_t= \sum_{k=0}^{\infty} \sum_{l_1,\ldots,l_k}
\beta_{k;l_1,\ldots,l_k} \xi_{t}^{i-l_1}\cdots\xi_{t}^{i-l_k}, $$
we have that this expansion converge in $L^2(\Omega')$, whenever
\begin{equation}\label{ec.6.3.5}
\sum_{k=0}^{\infty}\E\left[|\xi_{t}^{i}|^{2k}\right]
\sum_{l_1,\ldots,l_k} |\beta_{k;l_1,\ldots,l_k}|^2 < \infty.
\end{equation}
Then,
$$w_{n}^2 = \sum_{k=1}^{2n+1} \E\left[|\xi_{t}^{i}|^{2k}\right]
\sum_{-n=l_1<,\ldots,<l_k=n} |\beta_{k;l_1,\ldots,l_k}|^2,$$
\begin{equation*}
\begin{array}{llll}
\eta(r)^2 & \leq &   & \underset{k=0}{\overset{\infty}\sum}
\E[|\xi_{t}^{i}|^{2k}] \underset{l_1<\ldots<l_k< -r}{\sum} |\beta_{k;l_1,\ldots,l_k}|^2\\
          &      & + & \underset{k=0}{\overset{\infty}\sum}
\E[|\xi_{t}^{i}|^{2k}] \underset{l_1<\ldots<l_k< -r}{\sum}
|\beta_{k;l_1,\ldots,l_k}|^2.
\end{array}
\end{equation*}
So, under condition \eqref{ec.6.3.5}, we can also obtain that
$\eta(r)$ converges to zero as $r \rightarrow \infty$.

This case include noncausal $LARCH$, $ARCH$, $GARCH$ and bilinear
models.
\end{example}

\begin{example}[Causal $LARCH(\infty)$ innovations] \quad\\
general causal $LARCH(\infty)$ models are $\theta$-weakly
dependents. For instance, linear innovations, $ARCH$, $GARCH$ and
bilinear models, see \cite{Doukhan&Teyssiere&Winant2006}.
\end{example}


\subsubsection{Associated innovations}

\hspace{0.4cm} A process $\{\e^i\}$ is associated if
$$cov\left(f\left(\e^{i_1}, \ldots, \e^{i_n}\right),
g\left(\e^{i_1}, \ldots, \e^{i_n} \right)\right) \geq 0,$$
for any coordinate wise non-decreasing  function $f,g: \R^n
\rightarrow \R$. Associated processes or Gaussian stationary
processes are $\kappa$-weakly dependents with
$$\kappa(r)= \sup_{j\geq r} \left|cov\left(\e^i, \e^{i+j}\right)\right|.$$
For instance, independent sequence are associated and gaussian
processes with non-negative covariance are also associated. This
models are classically built from i.i.d sequence, see
\cite{Louhichi2000}.


\subsection{Uniform Lipschitz Bernoulli Shifts with weakly dependent innovations \label{sect.6.3.2}}

\hspace{0.4cm} Here, we consider the $DSULBS$ processes $Z^i= \{Z^i_t: \, t \in
\Z\}$ given by
$$Z^i_t= \mathcal{H}\left(\y^i, \{\varepsilon^i_{t-k}\}_{ k\in \mathbb{Z}} \right).$$
where $\Hc: \R^s \times \R^{\Z} \rightarrow \R$ be a measurable
function satisfying conditions~\eqref{ec.4.2.6} and
\eqref{ec.4.2.7}.

A simple example of this situation is the doubly stochastic linear
process. Another examples are obtained considering Lipschitz
function of linear processes.

In the following, we will consider that, for each $t$,
$\{\varepsilon^i_t: i \in \mathbb{Z}\}$ is $\epsilon$-weakly
dependent. Then, we prove that the sequence $Z$ of $DSULBS$
processes is in general a $(\epsilon,Y)$-weakly dependent doubly
stochastic processes. This is a new result and is obtained when the
$\epsilon$-coefficient is taking in the cases of $\lambda$, $\eta$,
$\theta$, $\kappa$ or $\kappa'$-weak dependence.

\begin{theorem}\label{theo.6.3.1} Under conditions \eqref{ec.4.2.6}
and \eqref{ec.4.2.7}, the $DSULBS$ processes $Z$ with $\epsilon$-weakly dependent innovation is
$(\epsilon,Y)$-weakly dependent, with $V(\y^{i})= \|a(\y^{i})\|_1$.
\end{theorem}


We will give the proof in Section~\ref{sect.6.6}.

\begin{remark}\label{rm.6.4.1}
If condition
\begin{description}
\item[E$2_\delta$] \qquad $\E[|\e^i_t|^{2+\delta}]< \infty$ for some $\delta>0$,
\end{description}
and condition~\eqref{ec.4.2.7} holds then
condition~K$2_{\delta}$ is satisfied for the $DSULBS$ process.
\end{remark}

\begin{remark}\label{rm.6.4.2}
Condition K5 is  satisfied by this model, since $V(\y)= \|a(\y)\|_1$ and so K5 is implied by condition~\eqref{ec.4.2.7}.
\end{remark}


\subsection{Doubly stochastic Volterra processes with weakly
dependent innovations \label{sect.6.3.3}}

\hspace{0.4cm} We will consider  $DVS^*$ processes defined by equation \eqref{ec.4.3.4}; i.e. we consider the Bernoulli shift
$$\Hc(\y,\e)= \sum_{k=0}^{\infty} \sum_{j_1<\ldots < j_k}
c_{k:j_1\ldots j_k}(\y) \e_{j_1}\ldots  \e_{j_k}.$$
Then, for all $\e, \e' \in \R^{\Z}$, we have that
\begin{equation}\label{ec.6.4.0}
\Hc(\y,\e) - \Hc(\y,\e')= \sum_{l \in \Z} \sum_{k=1}^{\infty}
\sum_{u=1}^k \sum_{\substack{j_1<\ldots<j_k \\ j_u=l }} c_{k:j_1,
\ldots, j_k} \e'_{j_1}\ldots
\e'_{j_{u-1}}(\e_{l}-\e'_{l})\e_{j_{u+1}}\ldots \e_{j_k}.
\end{equation}
We suppose that the $DVS^*$ processes $Z^i_t= \Hc(\y,
\{\e^i_{t-k}\}_{k\in\Z})$ are such that the sequence $c(\y)=\{c_{k:j_1,
\ldots, j_k}(\y): \, k\in \Z , \,(j_1,
\ldots, j_k) \in \Z^k \}$ satisfies condition C2, i.e.
\begin{description}
\item[C2]\qquad $\E[\|c(\y)\|_2^2]< \infty$.
\end{description}

Let us take $\varepsilon^i_t=
\mathcal{H}_{\varepsilon}\left(\{\xi_{t}^{i-l}\}_{l\in
\mathbb{Z}}\right)$ a  Bernoulli shift with independent inputs
$\{\xi_{t}^{l}\}$. We assume, in the same way of
Section~\ref{sect.6.3.1}, that the Bernoulli shift $\mathcal{H}_{\varepsilon}$ satisfies
condition~\eqref{ec.6.3.4}. Then, $\delta_r= \sum_{n \geq r} w_{n}$
converges to zero as $r \rightarrow \infty$.

In this section we assume the following condition
\begin{equation}\label{ec.6.4.1}
\E\left[\left|\e^{i,(s)}_t\right|^2\right]\leq \E\left[\left|\e^{i}_t\right|^2\right]=1.
\end{equation}

An adaptation of the proof of Lemma 3.1 and Lemma 3.2 given in \cite{WDB-PDoukhan}, allows us  to extend easily these results for doubly stochastic Volterra processes with Bernoulli
shift innovations.

\begin{theorem}\label{theo.6.4.1}
Under  conditions \eqref{ec.6.3.4}, \eqref{ec.6.4.1} and
condition~C2, the sequence $Z$ of $DSV^*$ processes with noncausal
Bernoulli shift innovations is $(\eta, Y)$-weakly dependent, with
$\eta(r)\leq 2\delta_{\lfloor r/2 \rfloor}$, and $V(\y^i)=\left\|c
(\y^{i})\right\|_2$.
\end{theorem}

\begin{theorem}\label{theo.6.4.2}
Under  conditions \eqref{ec.6.3.4}, \eqref{ec.6.4.1} and
condition~C2, the sequence $Z$ of $DSV^*$ processes with causal
Bernoulli shift innovations is $(\theta, Y)$-weakly dependent, with
$\theta(r)\leq 2\delta_{r}$ and $V(\y^i)= \|c(\y^i)\|_2$.
\end{theorem}

The proofs are deferred to Section~\ref{sect.6.6}.

From Theorem~\ref{theo.6.3.2} we have that the noncausal  Bernoulli
shift innovation with independent inputs is $\eta$-weakly dependent
and the causal Bernoulli shift innovation with independent inputs is
$\theta$-weakly dependent. Thus, from Theorem~\ref{theo.6.4.1} and
Theorem~\ref{theo.6.4.2}, we can confirmed that the weak dependence
property of innovations is transferred to the doubly stochastic
process.

\begin{remark}\label{rem.6.4.3}
Note that condition~\eqref{ec.6.3.4} and \eqref{ec.6.4.1} are
satisfied in the case of linear innovations or Volterra innovations,
see Example~\ref{ex.Linnov} and Example~\ref{ex.Vinnov}. So,
considering $DSV^*$ processes $Z^i$ such that the innovations
$\{\e^i\}$ are Volterra models with independent inputs, we have that
$Z$ is a sequence of weakly dependent doubly stochastic processes.
For instance, this is the case of $GARCH$, $ARCH(\infty)$,
$LARCH(\infty)$ and bilinear doubly stochastic processes.
\end{remark}

\begin{remark}\label{rm.6.4.4}
Condition K5 is satisfied by this type of doubly stochastic Volterra
models, in this case  $V(\y)= \|c(\y)\|_2$. So K5 is implied by
condition~C2.
\end{remark}


\section{Aggregation convergence results \label{sect.6.4}}

\hspace{0.4cm} We consider that $Z=\{Z^i\}$ is a weakly dependent stationary
sequence of doubly stochastic centered processes. Here, we consider
that conditions K$2_{\delta}$, K4 and K5 hold.

We now present our main results: the CLT, $\nu-a.s.$, for the
sequence $\{X^N(Y)\}$ in the case of $(\lambda,Y)$-weak dependence
and $(\kappa,Y)$-weak dependence.

\begin{theorem}[CLT: $(\lambda,Y)$-weak dependence]\label{theo.6.5.1}
We assume that $Z=\{Z^i\}$ is $(\lambda,Y)$-weakly dependent
satisfying conditions K$2_{\delta}$, K4, K5 and such that
$\lambda(r)= \mathcal{O}(r^{-\lambda})$, as $r \rightarrow \infty$,
for $\lambda
> 2+ \frac{3}{\delta}$. Then, $X^N(Y)$ converges in distribution,
$\nu-a.s.$, to a Gaussian process $X$ with covariance function
$\Gamma$.
\end{theorem}

\begin{theorem}[CLT: $(\kappa,Y)$-weak dependence]\label{theo.6.5.2}
We assume that $Z=\{Z^i\}$ is $(\kappa,Y)$-weakly dependent
satisfying conditions K$2_{\delta}$, K4, K5 and such that
$\kappa(r)= \mathcal{O}(r^{-\kappa})$, as $r \rightarrow \infty$,
for $\kappa > 2 + \frac{2}{\delta}$. Then, $X^N(Y)$ converges in
distribution, $\nu-a.s.$, to a gaussian process $X$ with covariance
function $\Gamma$.
\end{theorem}

\begin{remark}\quad\\ \vspace{-0.3cm}
\begin{itemize}
\item The result for $(\lambda,Y)$-weak dependence implies those for $(\eta,Y)$ or $(\theta,Y)$-weak dependence. In both cases, we do not achieve the better results, in the sense that the bound for the dependence parameters is also $2 + \frac{3}{\delta}$.

\item The results for $(\kappa',Y)$-weak dependence is
implied by case of $(\kappa,Y)$-weak dependence. In this case the
bound for these dependence parameters is $2 + \frac{2}{\delta}$.
\end{itemize}
\end{remark}
\vspace{0.3cm}

The proofs of these CLT  are extensions for the case doubly
stochastic processes of the proof of the CLT for weakly dependent sequences
given in (Section 7.1, \cite{WDB-PDoukhan}). The
proofs are deferred to Section~\ref{sect.6.6}.

On the other hand, we prove the following lemma, which implies that
under weak dependence property and condition E$2_{\delta}$, given in Remark~\ref{rm.6.4.1},
$\chi$ is a weak interaction in $\ell_1$.

\begin{lemma}\label{lem.6.5.0} Under condition E$2_{\delta}$:\\
If $\{\e^i\}$ is $\lambda$-weakly dependent then $|\chi(r)|\preceq \mathcal{O}(\lambda(r)^{\frac{\delta}{1+\delta}})$.\\
If $\{\e^i\}$ is $\kappa$-weakly dependent then $|\chi(r)|\preceq
\mathcal{O}(\kappa(r))$.
\end{lemma}

The proof of Lemma~\ref{lem.6.5.0} will be given in
Section~\ref{sect.6.6}. This result is essentially of
technical character and it allows us to get the SLLN for the
covariance function $\Gamma^N(Y)$.

\begin{remark}\label{rm.6.5.1}
In the case of $\lambda$-weak dependence, the
Lemma~\ref{lem.6.5.0} implies that $\chi(r)= \mathcal{O}
(r^{-\frac{\lambda \delta}{1+\delta}})$. Since, we suppose that
$\lambda> 2 + \frac{3}{\delta}$ then $\frac{\lambda
\delta}{1+\delta}> 1$, so $\chi \in \ell_1$.

In the cases of $\eta$, $\theta$-weak dependence we obtain a similar
result that  for the case of $\lambda$ weak dependence.
Nevertheless, in the case of $\kappa'$-weak dependence the result is
similar to case $\kappa$-weak dependence.
\end{remark}

\begin{remark}\label{rm.6.5.2}
In the case of elementary linear processes $Z^i$ with interactive linear innovation $\e^i$ given by $\varepsilon^i_t= \sum_{l \in \mathbb{Z}} \beta_l \xi_{t}^{i-l}$, where
$\|\beta\|_2^2 < \infty$ and $\{\xi_{t}^{l}\}$ a
sequence of i.i.d. random variables we have the following results.

As we have mentioned in Example~\ref{ex.Linnov}, $\varepsilon^i_t$ is $\eta$-weakly dependent with
$\eta(r) \leq \left(\sum_{|l|\geq r}\beta_l^2\right)^{\frac{1}{2}}$.

From Theorem~\ref{theo.6.3.2} and Remark~\ref{rem.6.4.3} we have that $Z= \{Z^i\}$ is $(\eta,Y)$-weakly dependent. Furthermore, from Remark~\ref{rm.6.4.4} condition $K5$ holds.

On the other hand,  from Remark~\ref{rm.6.5.2} we have that if condition K$2_{\delta}$ hold and $\eta(r)= \mathcal{O}(r^{-\eta})$ for $\eta > 2 + \frac{3}{\delta} $ then $\chi \in \ell_1$. So, from the SLLN given in \cite{Dacunha&Fermin.L.Notes} condition K4 holds.

Therefore, if condition K$2_{\delta}$ holds and $\eta(r)= \mathcal{O}(r^{-\eta})$ for $\eta > 2 + \frac{3}{\delta} $ then Theorem~\ref{theo.6.5.1} implies the $nu-a.s.$ weak convergence of $X^N(Y)$ to a Gaussian process $X$ with covariance function $\Gamma$.
\end{remark}


\section{An SLLN for $\Gamma^N(Y)$ in the case of $DSV^*$ processes \label{sect.6.5}}

\hspace{0.4cm} Now we give a SLLN for the covariance function $\Gamma^N(Y)$ of
$X^N(Y)$, in the case of  $DSV^*$ elementary processes defined by
equation \eqref{ec.4.3.4}. We
consider the case of interactive innovations, i.e. $\E[\e^i_t
\e^j_t]= \chi(i-j)$. In this case, the quadratic form $\Gamma^N(Y)$
is defined by
\begin{equation*}
\Gamma^N(\tau,Y)=\frac{1}{B_N}\sum_{i=1}^N \Psi_{\tau}(\y^i,\y^i)+
\frac{1}{B_N}\sum_{1\leq i \neq j\leq N}\Psi_{\tau}(\y^i,\y^j)
\; .
\end{equation*}
where
$$\Psi_{\tau}(\y^i,\y^j)= \sum_{k=1}^{\infty} \Psi_{\tau,k}(\y^i, \y^j) \chi^k(i-j), $$
with
$$\Psi_{\tau,k}(\y^i, \y^j)= \E^Y[Z^{i,(k)}_t Z^{j,(k)}_{t+\tau}]=
\sum_{l_1 < \ldots< l_k} c_{k: l_1, \ldots, l_k}(\y^i)
c_{k: l_1+\tau, \ldots, l_k+ \tau}(\y^j).$$

Let $\gamma_k(\tau)= \mathbb{E}[\Psi_{\tau,k}(\y^i, \y^i)]$ and
$\phi_k(\tau)= \mathbb{E}[\Psi_{\tau,k}(\y^i, \y^j)]$. Condition C2
implies that
$$\gamma(\tau)=\sum_{k=1}^\infty \gamma_k(\tau)< \infty \quad
\text{and}\quad  \phi(\tau)=\sum_{k=1}^\infty \phi_k(\tau)<
\infty.$$

We denote
$$[\chi^k]_{N,1}=\sum_{1\leq i\neq j \leq N}\chi^k(i-j), \quad \text{ and } \quad [|\chi|]_{N,1}=\sum_{1\leq i\neq j \leq N}|\chi(i-j)|,$$
where the function $\chi$ denotes the interaction between the
innovations. If $s_{i,k}=\sum_{j=1}^{i-1} \chi^k(j)$ then
$[\chi^k]_{N,1}=2\sum_{i=1}^N s_{i,k}$.

Given that, for all $k\geq 1$, $\chi^k$ is positive definite, we
have that $N\chi^k(0)+ [\chi^k]_{N,1} \geq 0$. Hence
$\frac{1}{N}[\chi^k]_{N,1}=\frac{2}{N}\sum_{i=1}^N s_{i,k} \geq -1$.

If $\chi \in \ell_1$ then, for all $k\geq 1$, $\chi^k \in \ell_1$ and
$$\frac{2}{N}\sum_{i=1}^N s_{i,k} = 2\sum_{i=1}^N \left( 1 - \frac{|i|}{N}\right) \chi^k(i) \limiteN 2 \sum_{i=1}^{\infty} \chi^k(i):= s_k.$$
So, $\{s_{i,k}: \, i \in \mathbb{N}\}$ converges in the Cesaro sense
to $\frac{1}{2}s_k$ with $-1 \leq s_k < \infty$.

\begin{theorem}[SLLN for $\Gamma^N(\tau)$: Volterra processes
case]\label{theo.2.3.3} If $\chi \in \ell_1$ and condition C2 holds, then
taking $B_N= \sqrt{N}$ we have that $\Gamma^N(Y)$ converge $\nu-a.s.$
and in $L^1(\nu)$ to $\Gamma$ given by
$$\Gamma(\tau)= \sum_{k=1}^{\infty} \gamma_k(\tau) + \sum_{k=1}^\infty \phi_k(\tau) s_k.$$
\end{theorem}

The proof of this theorem is an extension of SLLN given in \cite{Dacunha&Fermin.L.Notes} for the case of linear processes. The details are given in Section~\ref{sect.6.6}.


\section{Proof of the main results \label{sect.6.6}}


\subsection{Proof of Theorem~\ref{theo.6.3.1}}

\hspace{0.4cm} Without loss of generality we give the proof in the case of
$\lambda$-weakly dependent innovations, since in the other case the
 proof is similar.

\begin{proof}\qquad Let $i_1\leq, \ldots,\leq i_u \leq i_u + r < j_1 \leq, \ldots, \leq
j_v $
any $(u+v)$-tuples and $r>0$.\\
Let us consider $Z^{i}_t=\mathcal{H}\left(\y^{i},
\left\{\e^{i}_{t(k)} \right\}_{k}\right)$, with $\{t(k):\; k \in
\Z\}$ any sequence of indexes in $\Z$ and let us take
\begin{eqnarray*}
\varepsilon^{i,(s)}_{t(k) }= \varepsilon^{i}_{t(k)} \1_{\{|k| \leq
s\}}, & \quad & Z^{i,(s)}_t=\mathcal{H}\left(\y^{i}, \left\{\e^{i,
(s)}_{t(k)} \right\}_{k\in \Z}\right).
\end{eqnarray*}

For sake of simplicity, for all $f \in \Delta^{(u)}_1\,$ and $g \in
\Delta^{(v)}_1$, we denote $\mathbf{F}= f\left(Z^{i_1}_{t_1},
\ldots, Z^{i_u}_{t_u}\right)$, $\mathbf{G}=g\left(Z^{j_1}_{t'_1},
\ldots, Z^{j_v}_{t'_v}\right)$ and
\begin{equation}\label{ec.6.3.p1n1}
\begin{array}{lllll}
\mathbf{F}^{(s)} & = & F^{(s)}\left(\{\e^{i_1}_{t_1(k)} \}_{|k|\leq
s} ,\ldots, \{\e^{i_u}_{t_u(k)} \}_{|k|\leq s} \right)
& := & f\left(Z^{i_1, (s)}_{t_1}, \ldots, Z^{i_u,(s)}_{t_u}\right),\\
\mathbf{G}^{(s)} & = & G^{(s)}\left(\{\e^{j_1}_{t'_1(k)} \}_{|k|\leq
s} ,\ldots, \{\e^{j_v}_{t'_v(k)} \}_{|k|\leq s} \right) & := &
g\left(Z^{j_1, (s)}_{t'_1}, \ldots, Z^{j_v,(s)}_{t'_v}\right).
\end{array}
\end{equation}
It is easy to verify that
\begin{equation}\label{ec.6.3.p1n2}
Lip\left(F^{(s)}\right) \leq d_{\mathbf{i}_u}^s(Y) Lip(f), \quad
Lip\left(G^{(s)}\right) \leq d_{\mathbf{i}_v}^s(Y) Lip(g),
\end{equation}
where
$$ d^{(s)}_{\mathbf{i}_u}(Y)=  \sum_{m=1}^u \sum_{|k| \leq s} \left|a_{k}(\y^{i_m}) \right|
\quad \text{and}
 \quad d^{(s)}_{\mathbf{j}_v}(Y)=  \sum_{m=1}^v \sum_{|k| \leq s} \left|a_{k}(\y^{j_m}) \right|.$$

The proof proceeds in three parts:

\quad\\
\textbf{Part a:} First, we want to prove by means of an inductive
procedure the following result: considering $F^{(s)}$ and $G^{(s)}$
functions of types \eqref{ec.6.3.p1n1} such that
equations~\eqref{ec.6.3.p1n2} holds, then
$$\left|cov ( \mathbf{F}^{(s)}, \, \mathbf{G}^{(s)} |Y )\right|
\leq \psi\left(d^{(s)}_{\mathbf{i}_u}(Y), d^{(s)}_{\mathbf{i}_v}(Y),
Lip(f), Lip(g)\right ) \lambda(r),$$
where $\psi(d_{\mathbf{i}_u}, d_{\mathbf{j}_u}, a, b)= a
d_{\mathbf{i}_u} +  b d_{\mathbf{j}_u} + a b d_{\mathbf{i}_u}
d_{\mathbf{j}_u}$.

The proof of this part contains three steps.

\quad\\
\textbf{Step 1:} We verify the inductive hypothesis for $s=0$.

Let us $\mathcal{B}_{0}^{0}= \sigma\{\emptyset\}$ and
$\mathcal{B}_{0}^{m}= \sigma \left(
\left\{\varepsilon^{i_l}_{t_l(0)}, \varepsilon^{j_l}_{t'_l(0)}: \,
1\leq l \leq m \right\}\right) $ for $1 \leq m \leq u\wedge v$. If
we fix $\varepsilon^{i_l}_{t_l(0)}=e^{i_l}_0$,
$\varepsilon^{j_l}_{t'_l(0)}= e^{j_l}_0$ for $ 1\leq l \leq m$, then
we can write $\mathbf{F}^{(0)}$ and $\mathbf{G}^{(0)}$ respectively
as
\begin{equation*}
\begin{array}{lllll}
\mathbf{\tilde{F}}^{(0,m)}& := & \tilde{F}^{(0,m)}\left(
\e^{i_{m+1}}_{t_{m+1}(0)} ,\ldots, \e^{i_u}_{t_u(0)} \right) & = &
\mathbf{F}^{(0)} \left(e^{i_{1}}_0,\ldots, e^{i_{m}}_0,
\e^{i_{m+1}}_{t_{m+1}(0)} ,\ldots, \e^{i_u}_{t_u(0)} \right), \vspace{0.2cm} \\

\mathbf{\tilde{G}}^{(0,m)}& := &
\tilde{G}^{(0,m)}\left( \e^{j_{m+1}}_{t'_{m+1}(0)} ,\ldots,
\e^{j_v}_{t'_v(0)} \right) & = & \mathbf{G}^{(0)}
\left(e^{j_{1}}_0,\ldots, e^{j_{m}}_0, \e^{j_{m+1}}_{t'_{m+1}(0)}
,\ldots, \e^{j_v}_{t'_v(0)} \right).
\end{array}
\end{equation*}
Now, we introduce the following notation
\begin{equation*}
\begin{array}{lllll}
\mathbf{\bar{F}}^{(0,1)} & = & \bar{F}^{(0,1)}\left(
\e^{i_1}_{t_1(0)}\right)
& := & \mathbb{E}^Y\left[ \mathbf{F}^{(0)} \left| \mathcal{B}_{0}^{1}\right. \right], \vspace{0.2cm}\\

\mathbf{\bar{G}}^{(0,1)} & = & \bar{G}^{(0,1)}\left(
\e^{j_1}_{t'_1(0)}\right) &:= & \mathbb{E}^Y\left[ \mathbf{G}^{(0)}
\left| \mathcal{B}_{0}^{1}\right.\right],
\end{array}
\end{equation*}
and for $m>1$
\begin{equation*}
\begin{array}{lllll}
\mathbf{\bar{F}}^{(0,m)} & = & \bar{F}^{(0,m)}\left(
\e^{i_m}_{t_m(0)}\right)& := &
\mathbb{E}^Y\left[\left. \mathbf{\tilde{F}}^{(0,m-1)} \right| \mathcal{B}_{0}^{m}\right ],\vspace{0.2cm}\\

\mathbf{\bar{G}}^{(0,m)} & = & \bar{G}^{(0,m)}\left(
\e^{j_m}_{t'_m(0)}\right) &:= & \mathbb{E}^Y\left[\left.
\mathbf{\tilde{G}}^{(0,m-1)} \right| \mathcal{B}_{0}^{m}\right].
\end{array}
\end{equation*}

By using the expression for the conditional covariance given
$\mathcal{B}_{0}^{m}$,  for $Y$ fixed,  we have that
\begin{equation*}
\begin{array}{l}
\left|cov\left(\left.\mathbf{F}^{(0)}, \, \mathbf{G}^{(0)}\right|Y \right)\right|\\
\begin{array}{ll}
 \leq & \mathbb{E}^Y\left[ \left| cov \left(\left.\mathbf{F}^{(0)}, \,
\mathbf{G}^{(0)}\right|Y ,\,\mathcal{B}_{0}^{1}
\right)\right|\right] + \left| cov\left( \left.
\mathbb{E}^Y\left[\left.\mathbf{F}^{(0)}\right|
\mathcal{B}_{0}^{1}\right], \,
\mathbb{E}^Y\left[\left.\mathbf{G}^{(0)}\right| \mathcal{B}_{0}^{1}\right] \right| Y\right) \right| \vspace{0.2cm}\\
= & \mathbb{E}^Y\left[ \left| cov \left( \left.
\mathbf{\tilde{F}}^{(0,1)}, \, \mathbf{\tilde{G}}^{(0,1)} \right|
Y,\, \mathcal{B}_{0}^{1} \right) \right|\right] +
\left|cov\left(\left. \mathbf{\bar{F}}^{(0,1)} , \,
\mathbf{\bar{G}}^{(0,1)} \right|Y \right)\right|\\
\vdots& \\
\leq & \mathbb{E}^Y\left[ \left| cov \left( \left.
\mathbf{\tilde{F}}^{(0,m)}, \, \mathbf{\tilde{G}}^{(0,m)} \right|
Y,\, \mathcal{B}_{0}^{m} \right) \right|\right] + \sum_{l=1}^m
\mathbb{E}^Y\left[  \left|cov\left(\left. \mathbf{\bar{F}}^{(0,l)} ,
\, \mathbf{\bar{G}}^{(0,l)} \right|Y,\, \mathcal{B}_{0}^{l-1}
\right)\right| \right].
\end{array}
\end{array}
\end{equation*}
Following this procedure inductively until $m= u\wedge v$, we have
\begin{equation}\label{ec.6.3.p1n3}
\left|cov\left(\left.\mathbf{F}^{(0)}, \, \mathbf{G}^{(0)}\right|Y
\right)\right| \leq \sum_{m=1}^{u\wedge v} \mathbb{E}^Y\left[
\left|cov\left( \left. \mathbf{\bar{F}}^{(0,m)} , \,
\mathbf{\bar{G}}^{(0,m)} \right|Y,\, \mathcal{B}_{0}^{m-1}
\right)\right| \right].
\end{equation}

Since, for all $t$, $\{\e^{i}_{t}: \, i \in \mathbb{Z} \}$  is
$\lambda$-weakly dependent then
\begin{equation}\label{ec.6.3.p1n4}
\left|cov\left( \left.\mathbf{\bar{F}}^{(0,m)} , \,
\mathbf{\bar{G}}^{(0,m)} \right|Y, \,\mathcal{B}_{0}^{m-1}
\right)\right| \leq \psi \left(1,\, 1,\,
Lip\left(\bar{F}^{(0,m)}\right), \, Lip\left(\bar{G}^{(0,m)}\right)
\right) \lambda(r).
\end{equation}
Furthermore, for $Y$ fixed, we can verify
\begin{equation}\label{ec.6.3.p1n5}
\begin{array}{lll}
Lip\left(\bar{F}^{(0,m)}\right) \leq Lip(f)
\left|a_{0}\left(\y^{i_m}\right) \right|,
& \quad & \| \bar{F}^{(0,m)} \|_{\infty} \leq \|f\|_{\infty} \leq 1,\\
Lip\left(\bar{G}^{(0,m)}\right) \leq Lip(g)
\left|a_{0}\left(\y^{j_m}\right) \right|, & \quad & \|
\bar{G}^{(0,m)} \|_{\infty} \leq \|g\|_{\infty}\leq 1.
\end{array}
\end{equation}
Thus, from equations \eqref{ec.6.3.p1n3}, \eqref{ec.6.3.p1n4} and
\eqref{ec.6.3.p1n5}, we prove that
\begin{eqnarray*}
\left|cov\left(\left.\mathbf{F}^{(0)}, \, \mathbf{G}^{(0)}\right|Y
\right)\right| & \leq & \sum_{m=1}^{u\wedge v} \psi \left(1,\, 1,\,
Lip(\bar{F}^{(0,m)}), \, Lip(\bar{G}^{(0,m)})\right)
\lambda(r)\\
& \leq & \psi\left( d_{\mathbf{i}_u}^{(0)}     ,\,
d_{\mathbf{j}_v}^{(0)}   ,\, Lip(f) ,\, Lip(g)\right)\lambda(r).
\end{eqnarray*}

\quad\\
\textbf{Step 2:} We suppose the inductive hypothesis satisfied
for $0 \leq s < n$, i.e. if $F^{(s)}$ and $G^{(s)}$ are functions
defined in \eqref{ec.6.3.p1n1} satisfying
equations~\eqref{ec.6.3.p1n2} then
$$\left|cov \left(\left. \mathbf{F}^{(s)}, \, \mathbf{G}^{(s)}\right|Y \right)\right|
\leq \psi\left(d^{(s)}_{\mathbf{i}_u}(Y),\,
d^{(s)}_{\mathbf{j}_v}(Y),\, Lip(f),\, Lip(g) \right) \lambda(r).$$
\quad\\
\textbf{Step 3:} Now we will prove the inductive hypothesis for
$s=n$.

Let $\mathcal{B}_{n}= \sigma
\left(\left\{\varepsilon^{i_m}_{t_m(k)}, : \, |k|=n, \; m=1\ldots
u\right\} \cup \left\{\varepsilon^{j_m}_{t'_m(k)}, : \, |k|=n, \;
m=1\ldots v\right\}\right)$. If we fix, for $|k|=n$,
$\e^{i_m}_{t_m(k)}= e^{i_m}_{k}$  with $m=1 \ldots u$ and
$\e^{j_m}_{t'_m(k)}= e^{j_m}_{k}$ with $m=1 \ldots v$ then we can
write $\mathbf{F}^{(n)}$ and $\mathbf{G}^{(n)}$ respectively as
\begin{eqnarray*}
\mathbf{\tilde{F}}^{(n-1)}& := &
\tilde{F}^{(n-1)}\left(\{\e^{i_1}_{t_1(k)} \}_{|k|\leq n-1} ,\ldots,
\{\e^{i_u}_{t_u(k)}\}_{|k|\leq n-1}\right)= f\left(\tilde{Z}^{j_1,
(n-1)}_{t_1}, \ldots, \tilde{Z}^{j_u,(n-1)}_{t_u}\right),\vspace{0.2cm}\\
\mathbf{\tilde{G}}^{(n-1)}& := &
\tilde{G}^{(n-1)}\left(\{\e^{j_1}_{t'_1(k)} \}_{|k|\leq n-1},\ldots,
\{\e^{j_v}_{t'_v(k)}\}_{|k|\leq n-1}\right)= g\left(\tilde{Z}^{j_1,
(n-1)}_{t'_1}, \ldots, \tilde{Z}^{j_v,(n-1)}_{t'_v}\right),
\end{eqnarray*}
where $ \tilde{Z}^{i,(s)}_{t}= \Hc\left(\y^i,
\{\ldots,0,e^i_{-s},\e^{i}_{t(-s+1)},\ldots,\e^{i}_{t(s-1)},e^i_{s},
0, \ldots  \}\right)$.

Let us denote
\begin{equation*}
\begin{array}{lllll}
\mathbf{F}_{n} & = & F_{n}\left(
\{\e^{i_1}_{t_1(k)}\}_{|k|=n},\ldots,
\{\e^{i_u}_{t_u(k)}\}_{|k|=n}\right)
&:=& \mathbb{E}\left[\left. \mathbf{F}^{(n)} \right| \mathcal{B}_{n}\right ], \vspace{0.2cm}\\
\mathbf{G}_{n} & = & G_{n}\left(
\{\e^{j_1}_{t'_1(k)}\}_{|k|=n},\ldots,
\{\e^{j_v}_{t'_v(k)}\}_{|k|=n}\right)&:= &\mathbb{E}\left[\left.
\mathbf{G}^{(n)} \right| \mathcal{B}_{n} \right].
\end{array}
\end{equation*}
By using the expression for the conditional covariance given
$\mathcal{B}_{n}$,  for $Y$ fixed,  we have
\begin{equation}\label{ec.6.3.p1n6}
\begin{array}{l}
\left|cov\left(\left.\mathbf{F}^{(n)}, \, \mathbf{G}^{(n)}\right|Y \right)\right| \vspace{0.2cm}\\
\begin{array}{ll}
 \leq & \mathbb{E}^Y\left[ \left| cov \left(\left.\mathbf{F}^{(n)}, \,
\mathbf{G}^{(n)}\right|Y ,\mathcal{B}_{n} \right)\right|\right] +
cov\left( \left. \mathbb{E}^Y\left[\left. \mathbf{F}^{(n)}\right|
\mathcal{B}_{n}\right],
 \, \mathbb{E}^Y\left[\left.\mathbf{G}^{(n)}\right| \mathcal{B}_{n}\right] \right| Y \right) \vspace{0.2cm}\\
 = & \mathbb{E}^Y\left[ \left| cov \left( \left.
\mathbf{\tilde{F}}^{(n-1)}, \, \mathbf{\tilde{G}}^{(n-1)} \right| Y,
\mathcal{B}_{n} \right) \right|\right] + \left|cov\left(\left.
\mathbf{F}_{n} , \, \mathbf{G}_{n}\right|Y \right)\right|.
\end{array}
\end{array}
\end{equation}

It is easy to verify that
\begin{equation*}
\begin{array}{lllll}
Lip( \mathbf{\tilde{F}}^{(n-1)} )& \leq & Lip(f) \sum_{m=1}^u
\sum_{|k|\leq n-1} |a_k(\y^{i_m})| & = & Lip(f)d_{\mathbf{i}_u}^{n-1}(Y) \vspace{0.2cm}\\
Lip( \mathbf{\tilde{G}}^{(n-1)} )& \leq & Lip(g) \sum_{m=1}^u
\sum_{|k|\leq n-1} |a_k(\y^{j_m})| & = &
Lip(g)d_{\mathbf{j}_v}^{n-1}(Y).
\end{array}
\end{equation*}
Thus, applying the inductive hypothesis we obtain
\begin{equation}\label{ec.6.3.p1n7}
\mathbb{E}^Y\left[ \left| cov \left( \left.
\mathbf{\tilde{F}}^{(n-1)}, \, \mathbf{\tilde{G}}^{(n-1)} \right| Y,
\mathcal{B}_{n} \right) \right|\right]  \leq
\psi\left(d^{(n-1)}_{\mathbf{i}_u}(Y),\,
d^{(n-1)}_{\mathbf{j}_v}(Y),\, Lip(f), \,Lip(g)\right) \lambda(r).
\end{equation}

On the other hand, let us $\mathcal{B}_{n}^{+}=
\sigma\left(\left\{\e^{i_m}_{t_m(n)}: \, m=1\ldots u\right\} \cup
\left\{\e^{j_m}_{t'_m(n)}: \, m=1\ldots v\right\}\right)$. Now, if
we fix  $\e^{i_m}_{t_m(n)}= e^{i_m}_n$ for  $ m=1\ldots u$ and
$\e^{j_m}_{t'_m(n)} = e^{j_m}_n$ for $m=1\ldots v$, we can write
$\mathbf{F}_{n}$ and $\mathbf{G}_{n}$ respectively as
\begin{eqnarray*}
 \mathbf{F}_{-n}^{(0)}& = & F_{-n}^{(0)}\left(
\e^{i_1}_{t_1(-n)},\ldots, \e^{i_u}_{t_u(-n)}\right):= F_{n}\left(
\{ \e^{i_1}_{t_1(-n)}, e^{i_1}_n\},\ldots, \{\e^{i_u}_{t_u(-n)}, e^{i_u}_n \}\right) , \vspace{0.2cm}\\
 \mathbf{G}_{-n}^{(0)}& = & G_{-n}^{(0)}\left(
\e^{j_1}_{t'_1(-n)},\ldots, \e^{j_v}_{t'_v(-n)}\right):=
G_{n}\left(\{ \e^{j_1}_{t'_1(-n)}, e^{j_1}_n \},\ldots, \{
\e^{j_v}_{t'_v(-n)}, e^{j_v}_n\} \right).
\end{eqnarray*}
Using the notation
\begin{eqnarray*}
 \mathbf{F}_{n}^{(0)} & = &
F_{n}^{(0)}\left( \e^{i_1}_{t_1(n)},\ldots, \e^{i_u}_{t_u(n)}\right) := \E\left[ \left.\mathbf{F}_n \right| \mathcal{B}_{n}^{+}\right], \vspace{0.2cm}\\
 \mathbf{G}_{n}^{(0)} & = &
G_{n}^{(0)}\left( \e^{j_1}_{t'_1(n)},\ldots,
\e^{j_v}_{t'_v(n)}\right):= \E\left[\left. \mathbf{G}_n \right|
\mathcal{B}_{n}^{+}\right].
\end{eqnarray*}
We have
\begin{equation}\label{ec.6.3.p1n8}
\begin{array}{l}
 \left|cov\left(\left.\mathbf{F}_{n},\, \mathbf{G}_{n})\right|Y\right)\right|\vspace{0.2cm}\\
\begin{array}{ll}
 \leq & \mathbb{E}^Y\left[ \left| cov \left( \left. \mathbf{F}_{n}, \,
\mathbf{G}_{n}\right|Y ,\mathcal{B}_{n}^{+} \right) \right|\right] +
\left| cov\left( \left. \mathbb{E}^Y\left[\left.
\mathbf{F}_{n}\right| \mathcal{B}_{n}^{+}\right],
 \,\mathbb{E}^Y\left[\left. \mathbf{G}_{n} \right| \mathcal{B}_{n}^{+}\right] \right| Y\right) \right|\vspace{0.2cm}\\
 = & \mathbb{E}^Y\left[ \left| cov \left( \left.
\mathbf{F}_{-n}^{(0)}, \, \mathbf{G}_{-n}^{(0)} \right| Y,
\mathcal{B}_{n}^{+} \right) \right|\right] + \left| cov\left(\left.
\mathbf{F}_{n}^{(0)} , \, \mathbf{G}_{n}^{(0)}\right|Y
\right)\right|.
\end{array}
\end{array}
\end{equation}

Let us $\mathcal{B}_{k}^{0}= \sigma\{\emptyset\}$ and
$\mathcal{B}_{k}^{m}= \sigma \left(
\left\{\varepsilon^{i_l}_{t_l(k)}, \varepsilon^{j_l}_{t'_l(k)}: \,
1\leq l \leq m \right\}\right) $, for $1 \leq m \leq u\wedge v$ and
$|k|=n$. If we fix $\varepsilon^{i_l}_{t_l(k)}=e^{i_l}_k$,
$\varepsilon^{j_l}_{t'_l(k)}= e^{j_l}_k$ for $ 1\leq l \leq m$, then
we can write $\mathbf{F}^{(0)}_k$ and $\mathbf{G}^{(0)}_k$
respectively as
\begin{equation*}
\begin{array}{lllll}
\mathbf{\tilde{F}}^{(0,m)}_k& := & \tilde{F}^{(0,m)}_k\left(
\e^{i_{m+1}}_{t_{m+1}(k)} ,\ldots, \e^{i_u}_{t_u(k)} \right)
& = & \mathbf{F}^{(0)}_k \left(e^{i_{1}}_k,\ldots, e^{i_{m}}_k,  \e^{i_{m+1}}_{t_{m+1}(k)} ,\ldots, \e^{i_u}_{t_u(k)} \right),\vspace{0.2cm}\\
\mathbf{\tilde{G}}^{(0,m)}_k& := & \tilde{G}^{(0,m)}_k\left(
\e^{j_{m+1}}_{t'_{m+1}(k)} ,\ldots, \e^{j_v}_{t'_v(k)} \right) & = &
\mathbf{G}^{(0)}_k \left(e^{j_{1}}_k,\ldots, e^{j_{m}}_k,
\e^{j_{m+1}}_{t'_{m+1}(k)} ,\ldots, \e^{j_v}_{t'_v(k)} \right).
\end{array}
\end{equation*}
Let us denote
\begin{equation*}
\begin{array}{lllll}
\mathbf{\bar{F}}^{(0,1)}_k & = & \bar{F}^{(0,1)}_k\left(
\e^{i_1}_{t_1(k)}\right)
& := & \mathbb{E}^Y\left[ \mathbf{F}^{(0)}_k \,\left| \mathcal{B}_{k}^{1}\right. \right],\vspace{0.2cm}\\
\mathbf{\bar{G}}^{(0,1)}_k & = & \bar{G}^{(0,1)}_k\left(
\e^{j_1}_{t'_1(k)}\right) &:= & \mathbb{E}^Y\left[
\mathbf{G}^{(0)}_k \left| \mathcal{B}_{k}^{1}\right.\right],
\end{array}
\end{equation*}
and for $m>1$
\begin{equation*}
\begin{array}{lllll}
\mathbf{\bar{F}}^{(0,m)}_k & = & \bar{F}^{(0,m)}_k\left(
\e^{i_m}_{t_m(k)}\right)& := &
\mathbb{E}^Y\left[\left. \mathbf{\tilde{F}}^{(0,m-1)}_k \,\right| \mathcal{B}_{k}^{m}\right ],\vspace{0.2cm}\\
\mathbf{\bar{G}}^{(0,m)}_k & = & \bar{G}^{(0,m)}_k \left(
\e^{j_m}_{t'_m(k)}\right) &:= & \mathbb{E}^Y\left[\left.
\mathbf{\tilde{G}}^{(0,m-1)}_k \right| \mathcal{B}_{k}^{m}\right].
\end{array}
\end{equation*}
In the same way to Step 1, for $Y$ fixed, is obtain
\begin{equation*}
\begin{array}{lll}
Lip\left(\bar{F}^{(0,m)}_k\right) \leq Lip(f)
\left|a_{k}\left(\y^{i_m}\right) \right|,
& \quad & \| \bar{F}^{(0,m)}_k \|_{\infty} \leq \|f\|_{\infty} \leq 1, \vspace{0.2cm}\\
Lip\left(\bar{G}^{(0,m)}_k\right) \leq Lip(g)
\left|a_{k}\left(\y^{j_m}\right) \right|, & \quad & \|
\bar{G}^{(0,m)}_k \|_{\infty} \leq \|g\|_{\infty}\leq 1,
\end{array}
\end{equation*}
and
\begin{equation}\label{ec.6.3.p1n9}
\begin{array}{rll}
\left|cov\left(\left. \mathbf{F}_{-n}^{(0)} , \,
\mathbf{G}_{-n}^{(0)} \right| Y, \mathcal{B}_{n}^{+} \right)\right|
& \leq& \psi \left(d^{(0,
-n)}_{\mathbf{i}_u}(Y), \, d^{(0, -n)}_{\mathbf{j}_v}(Y),\, Lip(f) ,\, Lip(g) \right) \lambda(r)\vspace{0.2cm}\\
\left|cov\left( \left.\mathbf{F}_{\;n}^{(0)} , \,
\mathbf{G}_{\;n}^{(0)}\, \right| Y \right)\right| & \leq  &\psi
\left(d^{(0, n)}_{\mathbf{i}_u}(Y), \, d^{(0,
n)}_{\mathbf{j}_v}(Y),\, Lip(f) ,\, Lip(g)\right) \lambda(r),
\end{array}
\end{equation}
where $d_{\mathbf{i}_u}^{(0,k)}(Y)= \sum_{m=1}^u
\left|a_k(\y^{i_m})\right|$ and $d_{\mathbf{j}_v}^{(0,k)}(Y)=
\sum_{m=1}^v \left|a_k(\y^{j_m})\right|$.

Getting from equations \eqref{ec.6.3.p1n8} and \eqref{ec.6.3.p1n9},
\begin{equation}\label{ec.6.3.p1n10}
\left|cov\left( \mathbf{F}_{n} , \, \mathbf{G}_{n}| Y \right)\right|
\leq \psi\left(d^{(n)}_{\mathbf{i}_u}-d^{(n-1)}_{\mathbf{i}_u},
\,d^{(n)}_{\mathbf{j}_v}- d^{(n-1)}_{\mathbf{j}_v},\, Lip(f) ,\,
Lip(g)\right) \lambda(r).
\end{equation}
Finally, from \eqref{ec.6.3.p1n6}, \eqref{ec.6.3.p1n7} and
\eqref{ec.6.3.p1n10}, we have that
\begin{equation*}
\left|cov\left(\left.\mathbf{F}^{(n)}, \, \mathbf{G}^{(n)}\right|Y
\right)\right| \leq \psi\left(d^{(n)}_{\mathbf{i}_u},\,
d^{(n)}_{\mathbf{j}_v},\,Lip(f), \,Lip(g)\right) \lambda(r).
\end{equation*}

\quad\\
\textbf{Part b:} Taking $n\rightarrow \infty$ yields
$$|cov ( \mathbf{F}, \, \mathbf{G}|Y )|\leq \psi(d_{\mathbf{i}_u}(Y),\, d_{\mathbf{j}_v}(Y),\,Lip(f), \,Lip(g)) \lambda(r),$$
where
$$ d_{\mathbf{i}_u}(Y)=  \sum_{m=1}^u \left\|a(\y^{i_m})\right\|, \quad
\text{and} \quad d_{\mathbf{j}_v}(Y)=  \sum_{m=1}^v
\left\|a_{k}(\y^{j_m}) \right\|.$$
\quad\\
\textbf{Part c:} For $f,g \in \Delta^{(u)}_1\times\Delta^{(v)}_1$
and  $Z^{i}_t=\mathcal{H}\left(\y^{i}, \{\e^{i}_{t-k} \}_{k\in
\Z}\right)$ we denote
$$\mathbf{f}= f\left(Z^{i_1}_{t_1}, \ldots, Z^{i_u}_{t_u}\right),
\quad \text{and} \quad \mathbf{g}= g\left(Z^{j_1}_{t'_1}, \ldots,
Z^{j_v}_{t'_v}\right).$$
Then, reindexing the sequences $\{\e^i_{t-k}\}_{k \in \Z}$ by
$\{\e^i_{t(k)}\}_{k \in \Z}$ it holds
$$|cov ( \mathbf{f}, \, \mathbf{g}|Y )|\leq \psi(d_{\mathbf{i}_u}(Y),\, d_{\mathbf{j}_v}(Y),\,Lip(f), \,Lip(g)) \lambda(r).$$
Furthermore, we can verify that
$\E^Y\left[|Z^{i}_{t}|^2\right]^{\frac{1}{2}}\leq
\left\|a(\y^{i})\right\|= V(\y^i)$. Therefore, $Z$ is a
$(\lambda,Y)$-weakly dependent doubly stochastic process.
Analogously, the result can be proved for the cases of $\eta$,
$\theta$, $\kappa$ or $\kappa'$-weak dependence.
\end{proof}


\subsection{Proof of Theorem~\ref{theo.6.4.1}}

\hspace{0.4cm} This proof is  an adaptation for the case doubly stochastic of the
proof of Lemma 3.1 given in \cite{WDB-PDoukhan}.

\begin{proof} \quad Let $\xi_{k}^{i-l,(s)}= \xi_{k}^{i-l} \1_{|l|< s}$,
$\varepsilon_k^{i,(s)}= \mathcal{H}_{\varepsilon}\left(
\left\{\xi_{k}^{i-l,(s)}\right\}_{l\in \mathbb{Z}}\right)$ and
$Z^{i, (s)}_t= \mathcal{H}\left(\y^i, \{\varepsilon_{t-k}^{i,(s)}\}_{
k\in \mathbb{Z}} \right)$.\\
Let $f \in \Delta^{(u)}_1$ and $g \in \Delta^{(v)}_1$ with $u, v \in
\N$, and  $i_1\leq, \ldots,\leq i_u \leq i_u + r < j_1 \leq,
\ldots,\leq j_v $ with $r>2s$. We denote
\begin{equation*}
\begin{array}{lllllll}
\mathbf{f}& = & f(Z^{i_1}_{t_1}, \ldots, Z^{i_u}_{t_u}), & \quad&
\mathbf{f}^{(s)}& = & f(Z^{i_1, (s)}_{t_1}, \ldots, Z^{i_u,
(s)}_{t_u} ), \vspace{0.2cm}\\
\mathbf{g}& = & g(Z^{j_1}_{t'_1}, \ldots, Z^{j_v}_{t'_v}), & \quad&
\mathbf{g}^{(s)} & = & g(Z^{j_1, (s)}_{t'_1}, \ldots,Z^{j_v,
(s)}_{t'_v} ).
\end{array}
\end{equation*}

The sequences $\left\{\xi_{k}^{l, (s)} \right\}_{l \leq i}$ and $
\{\xi_{k'}^{l, (s)} \}_{l \geq i+r}$ are independent if $r>2s$. Then,
we have that $\mathbf{f}^{(s)}$ and $\mathbf{g}^{(s)}$ are
independent and thus
\begin{eqnarray*}
\left|cov \left(\mathbf{f}, \, \mathbf{g}|Y \right)\right| & \leq
\quad& \left|cov\left( \mathbf{f} - \mathbf{f}^{(s)} ,\, \mathbf{g}
|Y \right)\right|+
\left|cov\left( \mathbf{f}^{(s)}, \, \mathbf{g}- \mathbf{g}^{(s)}|Y \right)\right| \vspace{0.2cm}\\
&\leq \quad & 2 \|g\|_{\infty} \mathbb{E}^Y\left[ |\mathbf{f} -
\mathbf{f}^{(s)}|\right] + 2 \|f\|_{\infty}\mathbb{E}^Y\left[
| \mathbf{g} - \mathbf{g}^{(s)}|\right]\\
&\leq \quad& 2 Lip(f) \sum_{m=1}^{u}
\mathbb{E}^Y[|Z^{i_m}_{t_m} - Z^{i_m, (s)}_{t_m} |]\\
& \quad +& 2 Lip(g)\sum_{m=1}^{v} \mathbb{E}^Y[|Z^{j_m}_{t'_m} -
Z^{j_m, (s)}_{t'_m} |].
\end{eqnarray*}
We can verify that, for almost all $\y^i$, $\mathbb{E}^{\y^i}\left[|
Z^{i}_t |^2\right]^{\frac{1}{2}} = \left\|c(\y^i)\right\|_2$ and
conditions \eqref{ec.6.3.4}, \eqref{ec.6.4.1} imply
\begin{equation}\label{ec.6.4.2}
\begin{array}{l}
\mathbb{E}^{\y^i}[| Z^{i}_t - Z^{i, (s)}_t |]\\
\begin{array}{ll}
\leq & \mathbb{E}^{\y^i}[| Z^{i}_t - Z^{i, (s)}_t |^2]^{\frac{1}{2}}\\
%
%
\leq &\left( \underset{l \in \Z}\sum
\underset{k=1}{\overset{\infty}\sum} \underset{u=1}{\overset{k}\sum}
\underset{\substack{j_1<\ldots<j_k \\ j_u=l} }\sum  c^2_{k:j_1,
\ldots, j_k}(\y^i)
\right)^{\frac{1}{2}} \E[|\e^{i}_{t}-\e^{i,(s)}_{t}|^2]^{\frac{1}{2}}\\
\leq & \left\|c(\y^i)\right\|_2 \mathbb{E}\left[\left|
H_{\varepsilon}(\{\xi_{t}^{i-l}\}_{l \in \mathbb{Z}}) -
H_{\varepsilon}(\{\xi_{t}^{i-l}\}_{|l|<s }) \right|^2
\right]^{\frac{1}{2}}\\
\leq & \left\|c(\y^i)\right\|_2 \, \delta_s.
\end{array}
\end{array}
\end{equation}

From condition C2 we have that $\|c(\y)\|_2<\infty$ $\mu-a.s.$
Therefore, for $\eta(r) \leq 2 \delta_{\lfloor r/2 \rfloor}$ and
$V(\y)= \|c(\y)\|_2$ the result follows.
\end{proof}


\subsection{Proof of Theorem~\ref{theo.6.4.2}}

\hspace{0.4cm} This proof is  an adaptation for the case doubly stochastic of the
proof of Lemma 3.2 given in \cite{WDB-PDoukhan}.

\begin{proof} \quad Let $\xi_{k}^{j-l, (r)}= \xi_{k}^{j-l} \1_{\{l<r\}}$,
$\varepsilon_k^{j,(r)}= \mathcal{H}_{\varepsilon}\left(
\left\{\xi_{t}^{j-l, (r)}\right\}_{l\in \mathbb{Z}}\right)$ and
$Z^{j, (r)}_t=
\mathcal{H}\left(\y^j, \{\varepsilon_{t-k}^{j,(r)}\}_{ k\in \mathbb{Z}} \right)$.\\
Let $f \in \Delta^{(u)}_1$ and $g \in \Delta^{(v)}_1$ with $u, v \in
\N$, and  $i_1\leq, \ldots,\leq i_u \leq i_u + r < j_1 \leq,
\ldots,\leq j_v $ with $r>0$. 

We denote
\begin{equation*}
\begin{array}{lll}
\mathbf{f}& = & f(Z^{i_1}_{t_1}, \ldots, Z^{i_u}_{t_u}),\\
\mathbf{g}& = & g(Z^{j_1}_{t'_1}, \ldots, Z^{j_v}_{t'_v}),\\
\mathbf{g}^{(r)} & = & g(Z^{j_1, (r)}_{t'_1}, \ldots,Z^{j_v,
(r)}_{t'_v} ).
\end{array}
\end{equation*}

The sequences $ \{\xi_{k}^{l} \}_{l \leq i}$ and $
\left\{\xi_{k}^{l, (r)} \right\}_{l \geq i+r}$ are independent if
$r>0$. Then,  $\mathbf{f}$ and $\mathbf{g}^{(r)}$ are independent
and so
\begin{eqnarray*}
\left|cov\left( \mathbf{f}, \, \mathbf{g}\,| Y\right)\right| & = &
\left|cov\left( \mathbf{f}, \, \mathbf{g} -
\mathbf{g}^{(r)}\,| Y \right)\right|\\
& \leq & 2 \|f\|_{\infty}Lip(g)\sum_{m=1}^{v}
\mathbb{E}^Y[|Z^{j_m}_{t'_m} - Z^{j_m, (r)}_{t'_m}|].
\end{eqnarray*}
On the other hand, conditions \eqref{ec.6.3.4}, \eqref{ec.6.4.1}
imply $\mathbb{E}^{\y^j}[| Z^{j}_t - Z^{j, (s)}_t |]\leq
\|c(\y^{j}\|_2 \delta_r$. Furthermore, from condition~C2, we can
verify that $\mathbb{E}^{Y}\left[| Z^{i}_t |^2\right]^{\frac{1}{2}}
= \left\|c(\y^i)\right\|_2< \infty$ $\mu-a.s.$
Therefore, we obtain the result for $\theta(r) \leq 2 \delta_{r}$
and $V(\y)=\|c(\y)\|_2$.

\end{proof}


\subsection{Proof of Central Limit Theorems \ref{theo.6.5.1} and \ref{theo.6.5.2} }

\hspace{0.4cm} In this section we proof the Central Limit Theorems \ref{theo.6.5.1}
and \ref{theo.6.5.2} for the aggregation of weakly dependent doubly
stochastic processes under condition K$2_{\delta}$, K4 and
K5. The CLT are obtained by using Berstein's blocks arguments. This
proof is similar to proof of CLT given in (Section
7.1, \cite{WDB-PDoukhan}) but it is adapted in the context of
sequence weakly dependent of doubly stochastic processes.

Let $f(z)= \exp^{-i x z}$,  $f \in {\cal C}^3(\mathbb{R})$ with
bounded derivatives up to order 3. In the following, for $x,\, t$
fixed, we prove that
$$|\Delta_N|= |E^Y[f(X^N_t(Y))-f(X_t)]| \limiteN 0,$$
where $X^N(Y)$ is the partial aggregation process and $X$ is a
gaussian process with covariance function $\Gamma(\tau)$.

Let us consider three sequences of positives integers $p=\{p(N)\}_{N
\in \mathbb{N} }$, $q=\{q(N)\}_{N \in \mathbb{N} }$ and
$r=\{r(N)\}_{N \in \mathbb{N} }$ such that:
\begin{itemize}
\item $\lim_{N \rightarrow \infty } \frac{p(N)}{N} = \lim_{N \rightarrow \infty } \frac{q(N)}{p(N)}=0$.

\item $r(N)= \left[ \frac{N}{p(N)+q(N)}\right]$, thus $\lim_{N \rightarrow \infty }r(N)=\infty$.
\end{itemize}
These sequences are chosen to form the Berstein's blocks $I_1,...,
I_r$ and $J_1,..., J_r$ defined by:
\begin{eqnarray*}
I_m & = & \left\{ (m-1)(p(N)+q(N))+1,..., (m-1)(p(N)+q(N))+p(N) \right\}.\\
J_m & = & \left\{ (m-1)(p(N)+q(N))+p(N)+1,..., m(p(N)+q(N)) \right\}.\\
J_r & = & \left\{ r(p(N)+q(N))+1,..., N \right\}.
\end{eqnarray*}
Let $I = \bigcup_{m=1}^r I_m $, $ J = \bigcup_{m=1}^r J_m$ and
$U_m=\sum_{i \in I_m} Z^i_{t}$. Let ${\cal N}, \,{\cal N}_1,...,\,
{\cal N}_r$ be zero-mean Gaussian r.v. independent of the innovations
$\{\varepsilon^i\}$ and such that $\mathbb{E}[|{\cal N}_m|^2  ] =
\mathbb{E}^Y[|U_m|^2]$. We also consider a sequence $U^{*}_1,
\ldots, U^{*}_r$ of mutually independent r.v. such that, given $Y$,
$U^{*}_m$ has the same distribution as $U_m$. We take $B_N=
\sqrt{N}$, then $\Delta_N$ is decomposed as
\begin{equation}\label{ec.6.5.1}
\Delta_N(Y)= \sum_{l=1}^4 \Delta_{l,N}(Y),
\end{equation}
where
\begin{eqnarray*}
\Delta_{1,N}(Y) & = & E^Y\left[f\left( X^N_t(Y) \right)- f\left(\frac{1}{\sqrt{N}} \sum_{j=1}^r U_j\right)\right],\\
\Delta_{2,N}(Y) & = & E^Y\left[f\left(\frac{1}{\sqrt{N}}\sum_{j=1}^r U_j\right)- f\left(\frac{1}{\sqrt{N}}\sum_{j=1}^r U^{*}_j \right)\right],\\
\Delta_{3,N}(Y) & = & E^Y\left[f\left(\frac{1}{\sqrt{N}}\sum_{j=1}^r U^{*}_j\right)- f\left(\frac{1}{\sqrt{N}}\sum_{j=1}^r{\cal N}_j \right)\right],\\
\Delta_{4,N}(Y) & = & E^Y\left[f\left(\frac{1}{\sqrt{N}}\sum_{j=1}^r
{\cal N}_j \right) - f(\Gamma(0) {\cal N})\right].
\end{eqnarray*}

We now define a truncation procedure in order to be able to use
the previous weak dependence condition and applying the Lindenberg method.

For $T>1$, define $f_T(z)=(z\vee -T) \wedge T$, $z\in \mathbb{R}$.
Then $lip(f_T)=1$ and $\|f_T\|_{\infty}=T$.

\begin{lemma}\label{lemma.6.5.1} \quad $\mathbb{E}[|Z^i_t - f_T(Z^i_t)|^k]\leq   2\mathbb{E}[|Z^i_t|^m]T^{-(m-k)}$, for $k \, \in \mathbb{N}$ and $m>0$.
\end{lemma}

\begin{proof} \quad Applying the Holder's and Markov's inequalities yield

\begin{eqnarray*}
 \mathbb{E}[|Z^i_t -f_T (Z^i_t)|^k]
&\leq & \mathbb{E}[(|Z|^k + T^k)\1_{|Z|\geq T}]\\
&\leq & 2 \mathbb{E}[|Z|^k 1_{|Z|\geq T}]\\
&\leq & 2 \mathbb{E}[|Z|^m]^{\frac{k}{m}} P(|Z|\geq T)^{1-\frac{k}{m}}\\
&\leq & 2 \mathbb{E}[|Z|^m]^{\frac{k}{m}}\left(\frac{\mathbb{E}[|Z|^m]}{T^m}\right)^{1-\frac{k}{m}}\\
&\leq & 2 \mathbb{E}[|Z|^m]T^{-(m-k)}.
\end{eqnarray*}
\end{proof}



\begin{lemma}\label{lemma.6.5.2}\quad\\\vspace{-0.3cm}
If $Z$ is $(\lambda,Y)$-weakly dependent then
$$|cov(Z^i_t, Z^j_t)| \leq  (6 \mathbb{E}[|Z^i_t|^{2+\delta}] + 2(\mathbb{E}[V(\y)] + \mathbb{E}[V(\y)]^{2}))
\lambda(i-j)^{\frac{\delta}{1+ \delta}}.$$
If $Z$ is $(\kappa,Y)$-weakly dependent then
$$|cov(Z^i_t, Z^j_t)| \leq (6 \mathbb{E}[|Z^i_t|^{2+\delta}] + \mathbb{E}[V(\y)]^{2})) \kappa (i-j).$$
\end{lemma}

\begin{proof}
\begin{equation*}
\begin{array}{lllll}
|cov(Z^i_t, Z^j_t)| &\leq&   & |cov(Z^i_t - f_T(Z^i_t) ,\, Z^j_t )|    & \qquad \qquad (i)\\
                    &    & + & |cov(f_T(Z^i_t) ,\, Z^j_t - f_T(Z^j_t))|& \qquad \qquad (ii)\\
                    &    & + & |cov(f_T(Z^i_t) ,\, f_T(Z^j_t))|.        & \qquad \qquad (iii)
\end{array}
\end{equation*}

Let $m = 2+ \delta$ and $q=\frac{m}{m-1}$, then
$\frac{1}{m}+\frac{1}{q}=1$. The term $(i)$ is bounded applying
Lemma~\ref{lemma.6.5.1} and Holder's inequality
\begin{eqnarray*}
|cov(Z^i_t - f_T(Z^i_t) ,\, Z^j_t )| &\leq& \mathbb{E}[|(Z^i_t - f_T(Z^i_t)) Z^j_t |]\\
&\leq& \mathbb{E}[|Z^i_t - f_T(Z^i_t)|^{q}]^{\frac{1}{q}} \mathbb{E}[|Z^j_t |^m]^{\frac{1}{m}}\\
&\leq& (2\mathbb{E}[|Z^i_t|^m]T^{m-q} )^{\frac{1}{q}} \mathbb{E}^Y[|Z^j_t |^m]^{\frac{1}{m}}\\
&\leq& 2 \mathbb{E}[|Z^i_t|^m] T^{-(m-2)}.
\end{eqnarray*}
For the term $(ii)$, from Lemma~\ref{lemma.6.5.1} we have that
\begin{eqnarray*}
|cov(f_T(Z^i_t) ,\, Z^j_t - f_T(Z^j_t))| &\leq& 2 T \mathbb{E}[|Z^i_t - f_T(Z^i_t)|]\\
&\leq& 4 \mathbb{E}[|Z^i_t|^m] T^{-(m-2)}.
\end{eqnarray*}

We now consider a trajectory $Y$ fixed, then if $Z$ is
$(\lambda,Y)$-weakly dependent
\begin{eqnarray*}
|cov(f_T(Z^i_t) ,\, f_T(Z^j_t)|Y)| &\leq&
\left(T V(\y^i) + T V(\y^j)  + V(\y^i) V(\y^j) \right) \lambda(i-j).
\end{eqnarray*}
Taking expectation with respect to $Y$ and applying Holder's and
Jensen's inequalities  we can bound the $\mathbf{(iii)}$ term by
\begin{eqnarray*}
|cov(f_T(Z^i_t) ,\, f_T(Z^j_t))| &\leq& \mathbb{E}[|cov(f_T(Z^i_t) ,\, f_T(Z^j_t)|Y)| ]\\
&\leq& (2T\mathbb{E}[V(\y)] + \mathbb{E}[V(\y)]^{2}) \lambda(i-j)\\
&\leq& 2 T(\mathbb{E}[V(\y)] + \mathbb{E}[V(\y)]^{2}) \lambda(i-j).
\end{eqnarray*}
Finally, taking $T= \lambda(i-j)^{-\frac{1}{m-1}}$
\begin{eqnarray*}
|cov(Z^i_t ,\, Z^j_t )| &\leq&   (6\mathbb{E}[|Z^i_t|^m] T^{-(m-2)}
+ 2 T(\mathbb{E}[V(\y)] + \mathbb{E}[V(\y)]^{2}))
\lambda(i-j)\\
 &\leq&  (6 \mathbb{E}[|Z^i_t|^{2+\delta}] + 2(\mathbb{E}[V(\y)] + \mathbb{E}[V(\y)]^{2})) \lambda(i-j)^{\frac{\delta}{1+ \delta}}.
\end{eqnarray*}

In the case of $(\kappa,Y)$-weakly dependent, we have
\begin{eqnarray*}
|cov(f_T(Z^i_t) ,\, f_T(Z^j_t)|Y)| &\leq&
V(\y^i) V(\y^j) \kappa(i-j).
\end{eqnarray*}
Then, if we take $T = \kappa(i-j)^{-\frac{1}{m-2}}$, in a similar
way to the case of $(\lambda,Y)$-weakly dependent it yields
\begin{eqnarray*}
|cov(Z^i_t ,\, Z^j_t )| &\leq&  6\mathbb{E}[|Z^i_t|^m] T^{-(m-2)} + \E[V(\y)]^2\kappa (i-j)\\
 &\leq& \left(6\mathbb{E}[|Z^i_t|^{2+\delta}]+\E[V(\y)]^2 \right)  \kappa (i-j).
\end{eqnarray*}

\end{proof}

\begin{lemma}\label{lemma.6.5.3} \quad
$|\Delta_{1,N}(Y)|= \mathcal{O}\left(\left
(\frac{N-pr}{N}\right)^\frac{1}{2}\right)$ $\nu-a.s.$
\end{lemma}

\begin{proof} \quad Using Taylor's expansion up to the first order, we obtain
\begin{eqnarray*}
|\Delta_{1,N}(Y)| & \leq & \| f^{'} \|_{\infty} \frac{1}{\sqrt{N}} E^Y\left[\left|\sum_{i\in J } Z^i_{t}\right| \right]\\
& \leq & \| f^{'} \|_{\infty} \frac{1}{\sqrt{N}} \mathbb{E}^Y\left[\sum_{i, j \in J} Z^{i}_{t}Z^{j}_{t}\right]^{\frac{1}{2}}\\
& = & \| f' \|_{\infty} \left (\frac{N-pr}{N}\right)^\frac{1}{2}
\left(\frac{1}{N-pr} \sum_{i , j\in J } \Psi_0(\y^{i}, \y^{j})
\right)^{\frac{1}{2}}.
\end{eqnarray*}

On the other hand, since $Y$ is stationary then condition~K4 implies
\begin{eqnarray*}
\frac{1}{N-pr} \sum_{i, j\in J } \Psi_0(\y^{i}, \y^{j})
\xrightarrow{\; \nu-a.s. \; } \Gamma(0).
\end{eqnarray*}
%
%

%
%
%
%
%
Therefore, $\frac{1}{N-pr} \sum_{i , j\in J } \Psi_0(\y^{i}, \y^{j})=
\mathcal{O}(1)$ $\nu-a.s.$, and the result holds.

\end{proof}


\begin{lemma} \label{lemma.6.5.4}\quad
$|\Delta_{4,N}|=o(1)$ $\nu-a.s.$
\end{lemma}

\begin{proof}\quad Taylor's expansion up to the second order entails
\begin{eqnarray*}
|\Delta_{4,N}|
& \leq & \; \frac{1}{2} \| f'' \|_{\infty}\frac{rp}{N} \left|
\frac{1}{rp} \sum_{j=1}^r \mathbb{E}^Y[|\mathcal{U}_j|^2] -
\Gamma(0) \right| + \frac{1}{2} \| f'' \|_{\infty}
\frac{N-rp}{N}\Gamma(0).
\end{eqnarray*}
From condition K4 we have
$$\frac{1}{rp} \sum_{j=1}^r \mathbb{E}^Y[|U_j|^2]= \frac{1}{rp} \sum_{j=1}^r \sum_{i_1, i_2 \in I_j } \Psi_0(\y^{i_1},\y^{i_2})
\limiteas \Gamma(0).$$
Furthermore, $\frac{N-rp}{N} \preceq \frac{q}{p} \rightarrow 0$ as
$N \rightarrow \infty$, so that $|\Delta_{4,N}|= o(1)$.

\end{proof}


\begin{lemma} \label{lemma.6.5.5} \quad\\
In the $(\lambda,Y)$-weak dependence setting,
$$|\Delta_{2,N}|\preceq \left( \frac{r^2p}{\sqrt{N}}+ \frac{rp}{\sqrt{N}} + \frac{r^2p^2}{N}\right) \lambda(r) \quad \nu-a.s.$$
Moreover in the case of $(\kappa,Y)$-weak dependence
$$|\Delta_{2,N}|\preceq \left( \frac{r^2p^2}{N}\right)\kappa(r) \quad \nu-a.s.$$
\end{lemma}

\begin{proof} \quad Let define, for $j=1 \ldots r$,
\begin{eqnarray*}
\Lambda_j & = & \mathbb{E}^Y[f(\mathcal{W}_j + u_j) -
f(\mathcal{W}_j+ u^{*}_j)],
\end{eqnarray*}
where $ u_j= \frac{1}{\sqrt{N}}U_j$, $ u^{*}_j=
\frac{1}{\sqrt{N}}U^{*}_j$, $\mathcal{W}_j= \omega_j + \sum_{i>j}
u^{*}_i$ and $\omega_j = \sum_{i<j} u_i$.

Using the properties of the exponential function $f(z)= \exp^{-i x
z}$ and the independence properties of the variables $\{U^{*}_j\}$
we have
\begin{equation}\label{ec.6.5.4}
\begin{array}{lll}
 \Lambda_j   & = & \left(\mathbb{E}^Y[f(\omega_j)f(u_j)] - \mathbb{E}^Y[f(\omega_j)] \mathbb{E}^Y[f(u^{*}_j)] \right)
 \mathbb{E}^Y\left[f\left( \sum{i>j} u^{*}_i\right)\right],\vspace{0.2cm}\\
 & = & cov\left( f(\omega_j)\, , f(u_j)| Y\right) \mathbb{E}^Y\left[f\left( \sum_{i>j} u^{*}_i\right)\right].
\end{array}
\end{equation}
Then, we have
\begin{equation}\label{ec.6.5.5}
|\Delta_{2,N}|\leq \sum_{j=1}^r |\Lambda_j|.
\end{equation}
From equation~\eqref{ec.6.5.4} and applying the definition of
$(\lambda,Y)$-weak dependence, for $Y$ fixed, it follows that
\begin{equation*}
\begin{array}{lll}
|\Lambda_j| & \leq & \|f\|_{\infty} \left| cov\left( f\left(
\frac{1}{\sqrt{N}}\sum_{m<j} \sum_{i \in I_m} Z^i_t\right) ,\,
f\left(\frac{1}{\sqrt{N}} \sum_{i \in I_j} Z^i_t\right) \right)\right|\vspace{0.2cm}\\
& \leq & \|f\|_{\infty} \left[ lip(f) \|f\|_{\infty} \frac{1}{\sqrt{N}}\sum_{m<j}
\sum_{i_1 \in I_m} V(\y^{i_1}) \right.  +   lip(f) \|f\|_{\infty} \frac{1}{\sqrt{N}}\sum_{i_2 \in I_j} V(\y^{i_2})\vspace{0.2cm}\\
&      & +   \left. lip(f) lip(f) \frac{1}{\sqrt{N}}\sum_{m<j}
\sum_{i_1 \in I_m} V(\y^{i_1}) \frac{1}{\sqrt{N}}\sum_{i_2 \in I_j}
               V(\y^{i_2}) \right]\lambda(q).
\end{array}
\end{equation*}
Furthermore, by the  SLLN we have that condition K5 implies
$$\sum_{m<j} \sum_{i_1 \in I_m} V(\y^{i_1})= \mathcal{O}(rp), \quad \text{and} \quad  \sum_{i_2 \in I_j} V(\y^{i_2})=\mathcal{O}(p).$$
Then, from equation \eqref{ec.6.5.5}, the results follows for
$(\lambda,Y)$-weak dependence. In the case of $(\kappa,Y)$-weak
dependence the proof is similar.

\end{proof}

\begin{lemma} \label{lemma.6.5.6}\quad
$|\Delta_{3,N}|\preceq \frac{rp^{2+\delta}}{N^{1+
\frac{\delta}{2}}}$, $\nu-a.s.$
\end{lemma}

\begin{proof} \quad Let define, for $j=1 \ldots r$,
\begin{eqnarray*}
\Lambda^{'}_j & = &\mathbb{E}^Y[f(\mathcal{W}^{'}_j + u^{*}_j) -
f(\mathcal{W}^{'}_j+ u^{'}_j)],
\end{eqnarray*}
where  $ u^{*}_j= \frac{1}{\sqrt{N}}U^{*}_j$,
$u^{'}_j=\frac{1}{\sqrt{N}}\mathcal{N}_j$, $\mathcal{W}^{'}_j=
\sum{i<j} u^{*}_i + \sum_{i>j} u^{'}_i$.

Using the properties of exponential function $f(z)= \exp^{-i x z}$
and the independence properties of the variables
$\{\mathcal{N}_j\}$, we have
\begin{equation}\label{ec.6.5.6}
\Lambda^{'}_j  =   \mathbb{E}^Y[ f(u^{*}_j) - f( u^{'}_j)]
\mathbb{E}^Y[ f( \mathcal{W}^{'}_j)].
\end{equation}
Thus,
\begin{equation}\label{ec.6.5.7}
|\Delta_{3,N}|\leq \sum_{j=1}^r |\Lambda^{'}_j|.
\end{equation}

As $\mathbb{E}^Y[u^{*}_j]= \mathbb{E}^Y[u^{'}_j]= 0$ and
$\mathbb{E}^Y[|u^{*}_j|^2]=\mathbb{E}^Y[|u^{*}_j|^2]$, then taking
Taylor's expansion up to order 2 or 3 respectively  yield:
\begin{eqnarray*}
\mathbb{E}^Y[f(u^{*}_j)- f(u^{'}_j)]& = & \frac{1}{2}\mathbb{E}^Y\left[f^{(2)}(\theta^{*})(u^{*}_j)^2 - f^{(2)}(\theta^{'})(u^{'}_j)^2 \right],\\
\mathbb{E}^Y[f(u^{*}_j)- f(u^{'}_j)]& = &
\frac{1}{6}\mathbb{E}^Y\left[f^{(3)}(\theta^{*}_1)(u^{*}_j)^3 -
f^{(3)}(\theta^{'}_1)(u^{'}_j)^3\right].
\end{eqnarray*}
This implies that
\begin{equation}\label{ec.2.5.7}
\begin{array}{lll}
|\Lambda^{'}_j| & \leq &  \left(\frac{1}{2} \|f^{(2)}\|_{\infty}
\mathbb{E}[(u^{*}_j)^2 + (u^{'}_j)^2] \right)\wedge
                          \left(\frac{1}{6} \|f^{(3)}\|_{\infty} \mathbb{E}[(u^{*}_j)^3 + (u^{'}_j)^3] \right)\\
& \leq & \frac{1}{3^{\delta}}\|f\|_{\infty} \;
\|f^{(2)}\|^{1-\delta}_{\infty} \; \|f^{(3)}\|^{\delta}_{\infty} \;
\frac{1}{N^{1+\frac{\delta}{2}}}
\left(\mathbb{E}^Y[|U_j|^{2+\delta}] +
\mathbb{E}[|\mathcal{N}_j|^{2+\delta}]\right).
\end{array}
\end{equation}

Since $\{\mathcal{N}_j\}$ is a sequence of gaussian r.v.
\begin{eqnarray*}
\mathbb{E}[|\mathcal{N}_j|^{2+\delta}] & \leq &
3^{\frac{2+\delta}{4}} \mathbb{E}^Y[|U_j|^2]^{\frac{2+\delta}{2}}.
\end{eqnarray*}
On the other hand, applying Jensen's and Minkowski's inequalities, we get
\begin{equation*}
\mathbb{E}^Y[|U_j|^2]^{\frac{2+\delta}{2}} \leq
\mathbb{E}^Y[|U_j|^{2+\delta}] \leq \left( \sum_{i \in I_j}
\mathbb{E}^Y[|Z^i_t(\y^i)|^{2+\delta}]^{\frac{1}{2+\delta}}\right)^{2+\delta}.
\end{equation*}
From condition~K$2_{\delta}$ and applying the SLLN it comes
$$\mathbb{E}[|U_j|^{2+\delta}]+
\mathbb{E}^Y[|\mathcal{N}_j|^{2+\delta}]= \mathcal{O}(p^{2+
\delta}).$$ So, from \eqref{ec.6.5.6} and \eqref{ec.6.5.7} the
result holds.

\end{proof}
\vspace{0.3cm}

Let us assume that there exists $\alpha, \beta$ with $0 < \beta
<\alpha <1$ such that for $N \in \mathbb{N}$
$$p(N)=N^{\alpha}, \quad \mbox{and }\quad q(N)=N^{\beta}.$$
In the case of $(\lambda,Y)$-weak dependence: from Lemmas
\ref{lemma.6.5.3}, \ref{lemma.6.5.4}, \ref{lemma.6.5.5} and
\ref{lemma.6.5.6} we obtain that
\begin{eqnarray*}
\Delta_{1,N} &=& \mathcal{O}\left( N^{\frac{\beta - \alpha}{2}} + N^{\frac{\alpha - 1}{2}}\right),\\
\Delta_{2,N}& =& \mathcal{O}\left( N^{-\alpha - \lambda \beta +
\frac{3}{2}} + N^{-\lambda \beta +\frac{1}{2}}
                                + N^{-\lambda \beta + 1}\right),\\
\Delta_{3,N}& =& \mathcal{O}\left( N^{\alpha(1+\delta) - \frac{\delta}{2} } \right),\\
\Delta_{4,N} &=& o\left( 1 \right).
\end{eqnarray*}
Then, there exists $\theta>0$ such that $|\Delta_N|\leq N^{-\theta}$
if and only if
\begin{eqnarray} \label{ec.6.5.8}
0 < \beta < \alpha < \frac{\delta}{2(1+\delta)} \quad \text{and}
\quad \frac{3}{2} - \lambda \beta < \alpha.
\end{eqnarray}

We can to determine the values of $ \alpha$ and $\beta$ so that
condition~\eqref{ec.6.5.8} holds, whenever $\lambda> 2 +
\frac{3}{\delta}$. Moreover for the case of $(\kappa,2,Y)$-weak
dependence we readily obtain that there exists $\theta>0$ such that
$|\Delta_N|\leq N^{-\theta}$ whenever $\kappa> 2+ \frac{2}{\delta}$.

This allows to prove that, for all $t$ fixed, $\{X^N_t(Y)\}$ converge in
distribution to a mean zero Gaussian random variable with variance
$\Gamma(0)$.

Now, let us consider $\hat{X}^N(Y)= \sum_{k=1}^d \alpha_{k}
X^N_{t_k}(Y) $ with $(t_1,\ldots, t_d)\in \mathbb{Z}^d$ and
$(\alpha_{1}, \ldots , \alpha_{d}) \in \mathbb{R}^d $. We can
write $\hat{X}^N$ as
$$\hat{X}^N(Y)= \frac{1}{\sqrt{N}} \sum_{i=1}^N \hat{Z}^i,$$
where $\hat{Z}^i= \sum_{k=1}^d \alpha_{k} Z^i_{t_k}$.

Using a similar technique to the one applied in the proof of
Theorem~\ref{theo.6.3.1}, we can verify easily that if $Z= \{Z^i\}$ is
a $(\epsilon,Y)$-weakly dependent doubly stochastic processes then
$\{\hat{Z}^i\}$ is also $(\epsilon,Y)$-weakly dependent. Moreover,
condition K$2_{\delta}$, K4 and K5 are satisfied for
$\{\hat{Z}^i\}$. Furthermore, the variance $\sigma_N^2(Y)$ of
$\hat{X}^N(Y)$ is such that
$$
\sigma_N^2(Y)=  \sum_{k,l=1}^{d} \alpha_{k}
\Gamma^N(t_k-t_l) \alpha_{l} \xrightarrow{\quad \nu-a.s. \quad}
\sum_{k,l=1}^{d} \alpha_{k} \Gamma(t_k-t_l)
\alpha_{l}= \sigma^2.
$$

Then, in the same way as before we show that $\hat{X}^N(Y)$ converge weakly,
$\nu-a.s.$, to a centered Gaussian random variable
with variance $\sigma^2$. Therefore, we obtain that the process
$X^N(Y)=\{X^N_t(Y): \, t \in\mathbb{Z}\}$ converges in distribution,
$\nu-a.s.$,  to a centered Gaussian process $X= \{X_t: \, t
\in\mathbb{Z}\}$ with covariance function $\Gamma$.

\subsection{Proof of Lemma~\ref{lem.6.5.0}}

\hspace{0.4cm} We prove, under weak dependence property and condition E$2_{\delta}$, that
there exist $d>1$ such that $|\chi(r)|\preceq O(r^{-d})$, so $\chi$
is a weak interaction in $\ell^1$.
\\

\begin{proof}\qquad Let $T>1$, we define $f_T(z)=(z\vee -T) \wedge T$, $z\in \mathbb{R}$,  $lip(f_T)=1$ and
$\|f_T\|_{\infty}=T$. By Lemma~\ref{lemma.6.5.1},  for $k \, \in
\mathbb{N}$, we have
$$\mathbb{E}[|\e^i_t - f_T(\e^i_t)|^k]\leq   2\mathbb{E}[|\e^i_t|^{2+\delta}]T^{-(2+\delta-k)}.$$
Then if $\{\e^i_t \}$ is $\lambda$-weakly dependent, following the
proof of Lemma~\ref{lemma.6.5.2}, we have
\begin{equation*}
\begin{array}{lll}
|\chi(i-j)| &\leq&    |cov(\e^i_t - f_T(\e^i_t) ,\, \e^j_t )|    \\
            &    & +  |cov(f_T(\e^i_t) ,\, \e^j_t - f_T(\e^j_t))|\\
            &    & +  |cov(f_T(\e^i_t) ,\, f_T(\e^j_t))|         \\
            &\leq& 6\mathbb{E}[|\e^i_t|^{2+\delta}]T^{-\delta}+(2T+1)\lambda(i-j)\\
            &\leq& (6 \vee \mathbb{E}[|\e^i_t|^{2+\delta}])(T^{-\delta}+ T\lambda(i-j)).
\end{array}
\end{equation*}
Finally, taking $T= \lambda(i-j)^{\frac{1}{1+\delta}}$, we obtain
that $|\chi(i-j)|\preceq O(\lambda(i-j)^{\frac{\delta}{1+\delta}})$.

Otherwise in the case of $\kappa$-weakly dependent, we obtain
\begin{equation*}
|\chi(i-j)| \leq 6\mathbb{E}[|\e^i_t|^{2+\delta}]T^{-\delta}+
\kappa(i-j).
\end{equation*}
Then, by choosing $T$ such that $ \kappa(i-j)=
6\mathbb{E}[|\e^i_t|^{2+\delta}]T^{-\delta}$, it follows
$|\chi(i-j)|\preceq O(\kappa(i-j))$.

\end{proof}


\subsection{Proof of the SLLN for $\Gamma^N(Y)$: case of $DSV^*$ processes.}

\hspace{0.4cm} The following proof goes to the same lines of the one given in
\cite{Dacunha&Fermin.L.Notes} for SLLN in the case of linear processes.

\begin{proof} \qquad The proof of theorem will be presented in three parts.

\quad\\
\textbf{Part 1:} Since $\chi \in
\ell_1$, then $\frac{[\chi^k]_{N,1}}{N}\leq \|\chi\|_1$ and
$\frac{[\chi^k]_{N,1}}{N}$ converge to $s_k$, for each $k$, with
$|s_k| \leq \|\chi\|_1$. On the other hand, condition~C2 implies
$\sum_{k=1}^\infty \phi_k(\tau)< \infty$. Then, for all $\tau \in
\mathbb{Z}$, we have
$$\sum_{k=1}^\infty \phi_k(\tau)\frac{[\chi^k]_{N,1}}{N} \limiteN  \sum_{k=1}^\infty \phi_k(\tau) s_k < \infty.$$
Therefore,
$$R_N(\tau):=\mathbb{E}[\Gamma^N(\tau,Y)]= \sum_{k=1}^\infty
\gamma_k(\tau) + \sum_{k=1}^\infty
\phi_k(\tau)\frac{[\chi^k]_{N,1}}{N} \limiteN \Gamma(\tau).$$
For any values of sequence $\{s_k\}$, $R_N(\tau)$ has a non-zero
limit.

\quad\\
\textbf{ Part 2:}  Let $M_{\tau,k}(\y^i) =
\mathbb{E}^i[\Psi_{\tau,k}(\y^i,\y^j)]$, for $i \neq j$, where
$\mathbb{E}^i[\cdot]$ is the conditional expectation with respect to
$\y^i$. Then, $\{M_{\tau,k}(\y^i): \, i \in \mathbb{N}\}$ is an i.i.d.
sequence with $\mathbb{E}[M_{\tau,k}(\y^i)]= \phi_k(\tau)$.

Let $\mathbb{H}$ be the Hilbert space generated by
$\{\Psi_{\tau,k}(\y^i,\y^j): \, k \in \mathbb{N}, \, 1\leq i \neq j
\leq N\}$, $\mathbb{H}_1$ the linear space generated by
$\{\Psi_{\tau,k}(\y^i,\y^j)-M_{\tau,k}(\y^i)-M_{\tau,k}(\y^j)+\phi_k(\tau):\,
k \in \mathbb{N}, \, 1\leq i<j \leq N\}$, $\mathbb{H}_2$ the linear
space generated by $\{M_{\tau,k}(\y^i)+M_{\tau,k}(\y^j)-2\phi_k(\tau):
\, k \in \mathbb{N}, \, 1\leq i<j \leq N\}$ and $\mathbb{C}$ the
space of constants, then $\mathbb{H}_1$, $\mathbb{H}_2$,
$C$ form an orthogonal decomposition of $\mathbb{H}$; i.e.
$\mathbb{H}=\mathbb{H}_1 \bigoplus\mathbb{H}_2 \bigoplus C$.
This can be checked by realizing that
$\mathbb{E}^i[\Psi_{\tau,k}(\y^i,\y^j)]=0$ $\mu-a.s.$ for $i \neq j$.
We define
\begin{eqnarray*}
T_N(\tau,Y)& = & \frac{1}{N}\sum_{i=1}^N
\sum_{k=1}^{\infty}\left(\Psi_{\tau,k}(\y^i,\y^i)-\gamma_k(\tau)\right)
\; ,\\
Q_N(\tau,Y)& = & \frac{1}{N}\sum_{1\leq i\neq j\leq N}
\sum_{k=1}^{\infty}\left(M_{\tau,k}(\y^i)+M_{\tau,k}(\y^j)-2\phi_k(\tau)\right)\chi^k(i-j)
\; ,\\
U_N(\tau,Y)& = & \frac{1}{N}\sum_{1\leq i\neq j\leq N}
\sum_{k=1}^{\infty}\left(\Psi_{\tau,k}(\y^i,\y^j)-M_{\tau,k}(\y^i)-M_{\tau,k}(\y^j)+\phi_k(\tau)\right)\chi^k(i-j)
\; .
\end{eqnarray*}
$\mathbb{H}$'s orthogonal decomposition applied to
$\Gamma^N(\tau,Y)-T_N(\tau,Y)$ gives the following orthogonal
decomposition
$$\Gamma^N(\tau,Y)-T_N(\tau,Y)= R_N(\tau,Y) +
Q_N(\tau,Y) + U_N(\tau,Y) \; .$$

In what follows, we will show that $T_N(\tau,Y)$, $Q_N(\tau,Y)$ and
$U_N(\tau,Y)$ converge to zero $\nu-a.s$ and in $L^1(\nu)$ under
some given conditions. If the limits $\nu-a.s$ and in $L^1(\nu)$ of
$\Gamma^N$ exist, then they must be the same.

\quad\\
\textbf{ Step 1:} ($T_N$'s convergence to zero).\\
Under condition~C2 we have
$$\mathbb{E}\left[\left|\sum_{k=1}^\infty \Psi_{\tau, k}(\y^i,
\y^i)\right|\right]\leq \mathbb{E}\left[\|c(\y^i)\|^2_2 \right] <
\infty.
$$
Then,, for each $\tau \in \mathbb{Z}$, the SLLN implies
that $T_N(\tau,Y)$ converges in $L^1(\nu)$ and $\nu-a.s.$ to zero.

\quad\\
\textbf{Step 2:} ($Q_N$'s convergence to zero).\\
We can write
$$Q_N(\tau,Y)= \frac{2}{N}\sum_{i=1}^{N} \sum_{k=1}^\infty \left(M_{\tau,k}(\y^i)-\phi_k(\tau)\right)\left(s_{N+1-i}(k)+s_{i}(k)\right),$$
So that $\mathbb{E}[Q_N(\tau,Y)] = 0$ and
\begin{eqnarray*}
\mathbb{E}[|Q_N(\tau,Y)|^2] & = &  \frac{\Delta_N}{N^2} .
\end{eqnarray*}
where
\begin{eqnarray*}
\Delta_N &=& 4\sum_{i=1}^N \sum_{k,l =1}^{\infty} cov\left(M_{\tau,k}(\y^i), \,
M_{\tau,l}(\y^i)\right) \left(s_{N-i}(k)+s_{i}(k)\right)\left(s_{N-i}(l)+s_{i}(l)\right)\\
&\leq & 4\sum_{i=1}^N \sum_{k,l =1}^{\infty} A_{\tau,k}
A_{\tau,l}\left|s_{N-i}(k)+s_{i}(k)\right|\left|s_{N-i}(l)+s_{i}(l)\right|,
\end{eqnarray*}
with $ A_{\tau,k} = \mathbb{E}\left[|M_{\tau,k}(\y^i)-\phi_k(\tau)|^2\right]^{\frac{1}{2}}$.
Furthermore,
\begin{eqnarray*}
\left|s_{N-i}(l)+s_{i}(l)\right| & \leq &
\sum_{j=1}^{N-i-1}|\chi^l(j)| + \sum_{j=1}^{i}|\chi^l(j)| \leq 2
\|\chi^l\|_1 \leq 2 \|\chi\|_1\\
\frac{1}{N^2}\sum_{i=1}^N \left|s_{N-i}(k)+s_{i}(k)\right| &
\preceq& \frac{[|\chi|]_{N,1}}{N^2}.
\end{eqnarray*}
Moreover, from condition~C2
\begin{eqnarray*}
\sum_{k=1}^{\infty} A_{\tau,k} & \leq & \sum_{k=1}^{\infty}
\mathbb{E} \left[|\Psi_{\tau,k}(\y^i, \y^j) -
\phi_k(\tau)|^2\right]^{\frac{1}{2}}\\
& \leq &  \sum_{k=1}^{\infty} \mathbb{E} \left[|\Psi_{\tau,k}(\y^i,
\y^j)|^2\right]^{\frac{1}{2}}\\
& \leq & \sum_{k=0}^{\infty} \sum_{l_1<\ldots<l_k}
\mathbb{E}[|c_{k;l_1,\ldots,l_k}(\y)|^2]^{\frac{1}{2}}
\mathbb{E}[|c_{k;l_1+\tau,\ldots,l_k+\tau}(\y)|^2]^{\frac{1}{2}}\\
& \leq & \mathbb{E}\left[ \|c(\y)\|^2_2 \right] < \infty.
\end{eqnarray*}
Therefore,
\begin{equation}\label{ec.6.5.8.1}
\Delta_N  \preceq 8 \|\chi\|_1\left(\sum_{k=1}^{\infty} A_{\tau,k}\right)^2 [|\chi|]_{N,1}\quad \text{ and }\quad  \mathbb{E}[|Q_N(\tau,Y)|^2]  \preceq \mathcal{O}\left(\frac{[|\chi|]_{N,1}}{N^2}\right).
\end{equation}
Thus, we obtain that $\sum_N \mathbb{E}[Q_N(\tau,Y)]^2 < \infty$. Since  $[|\chi|]_{N,1}= \mathcal{O}(N)$ then $Q_N(\tau,Y)$ converges to zero in $L^2(\nu)$.

We have two cases:
\begin{itemize}
\item Case 1:  $\Delta_N$ converges to a finite limit,  $\Delta_N= \mathcal{O}(1)$.

\item Case 2:  $\Delta_N$ converges to $+\infty$.
\end{itemize}

In the case 1, we have $\mathbb{E}[|Q_N(\tau,Y)|^2] = \mathcal{O}(N^{-2})$, so from Borel-Cantelli's lemma we derive the $\nu-a.s.$ convergence to zero of $Q_N(\tau,Y)$.

In the case 2, to prove $\lim_{N\rightarrow \infty} Q_N(\tau,Y)=0$
$\nu-a.s.$, we can apply Petrov's Theorem (\cite{Petrov},6.17 p
222): let $Q^*_N=\sum_{i=1}^N \xi^i$ be a sum of independent centered
random variables such that its variance $\Delta_N$ diverge
to infinity. Then $Q^*_N =o(\sqrt{\Delta_N\Theta(\Delta_N)})$ for all
function $\Theta$ such that $\sum\frac{1}{n\Theta(n)}<\infty$.

In our case, $Q^N= Q^*_N/ N$, so it is sufficient to find $\Theta$ such that
$$\Delta_N \Theta(\Delta_N)= \mathcal{O}(N^2).$$
We have from \eqref{ec.6.5.8.1} that $\Delta_N = \mathcal{O}([|\chi|]_{N,1})$ and $[|\chi|]_{N,1})=\mathcal{O}(N)$. Then, taking $\Theta(n)=n$ we obtain
$$\Delta_N \Theta(\Delta_N)= \Delta_N^2 = \mathcal{O}([|\chi|]^2_{N,1})= \mathcal{O}(N^2).$$
Consequently $Q_N(\tau,Y)$ converges $\nu-a.s.$ to zero.

\quad\\
\textbf{Step 3:} ($U_N$'s convergence to zero).\\
We consider the kernel defined, for $i\neq j$, by
\begin{eqnarray*}
\Phi_{\tau}(\y^i,\y^j)& = & \sum_{k=1}^{\infty}\Phi_{\tau,k}(\y^i,\y^j)\chi^k(i-j),
\end{eqnarray*}
where
\begin{eqnarray*}
\Phi_{\tau,k}(\y^i,\y^j)& = &
\Psi_{\tau,k}(\y^i,\y^j)-M_{\tau,k}(\y^i)-M_{\tau,k}(\y^j)+\phi_k(\tau)
\; .
\end{eqnarray*}
This kernel is symmetric and degenerated, i.e.,
$\mathbb{E}^j[\Phi_{\tau}(\y^i,\y^j)]=0$  $\mu-a.s.$ Whence we
have that $\mathbb{E}[\Phi_{\tau}(\y^i,\y^j)]=0$,
$\mathbb{E}[\Phi_{\tau}(\y^i,\y^j)\Phi_{\tau}(\y^m,\y^n)]=0$ for
$(i,j)\neq (m,n)$.

In the same way as we bound $\sum_k A_{\tau,k}$ in the Step 2, we
can verify that condition~C2 implies
$$\sigma_{\tau}=\sum_{k=1}^{\infty} \mathbb{E}[|\Phi_{\tau,k}(\y^i,\y^j)|^2]^{\frac{1}{2}} < \infty .$$

Then writing
$$U_N(\tau,Y)= \frac{1}{N} \sum_{1\leq i \neq j \leq N} \Phi_{\tau}(\y^i,\y^j),$$
we obtain
\begin{eqnarray*}\label{ec.6.5.8.2}
\mathbb{E}[|U_N(\tau,Y)|^2]& = & \frac{1}{N^2} \sum_{1\leq i \neq j
\leq N} \sum_{k,l =1}^{\infty}
\mathbb{E}[\Phi_{\tau,k}(\y^i,\y^j)\Phi_{\tau,l}(\y^i,\y^j)] \chi(i-j)^k
\chi(i-j)^l\\
& \leq & \frac{1}{N^2} \sum_{1\leq i \neq j \leq N} \sum_{k,l
=1}^{\infty} \mathbb{E}[|\Phi_{\tau,k}(\y^i,\y^j)|^2]^{\frac{1}{2}}
\mathbb{E}[|\Phi_{\tau,l}(\y^i,\y^j)|^2]^{\frac{1}{2}} |\chi^k(i-j)||\chi^l(i-j)|\\
& \leq &\left(\sum_{l =1}^{\infty}
\mathbb{E}[|\Phi_{\tau,l}(\y^i,\y^j)|^2]^{\frac{1}{2}}
\left(\frac{1}{N^2}[\chi^{2k}]_{N,1}\right)^{\frac{1}{2}} \right)^2\\
& \preceq & \sigma^2_{\tau} \frac{[|\chi|]_{N,1}}{N^2}.
\end{eqnarray*}
Hence $\sum_N \mathbb{E}[U_N(\tau,Y)]^2 < \infty$. Since
$[|\chi|]_{N,1}= \mathcal{O}(N)$ then $Q_N(\tau,Y)$ converges to
zero in $L^2(\nu)$.

Let us  now prove  the $\nu-a.s.$ convergence of $U_N$. We prove the $\nu-a.s.$ convergence of $U_N$ to zero, following the scheme of the classical proof for the SLLN in
the case of i.i.d., \cite{Petrov}. See \cite{Surgailis}
for the Central Limit Theorem.

Let $a>0$,
\begin{eqnarray*}
\mathbb{P}\left(\max_{n\geq  N}|U_n(\tau,Y)|   \geq  2a\right)  &
\leq & \!\! \sum_{k=\lfloor\sqrt{N}\rfloor}^{\infty}
\!\!\mathbb{P}\left(|U_{k^2}(\tau,Y)|    \geq    a \right)\\
& +    & \!\! \sum_{k=\lfloor\sqrt{N}\rfloor}^{\infty}      \!\!
\mathbb{P}\!\! \left(\max_{k^2\leq n \leq (k+1)^2}\!\!
|U_n(\tau,Y)-U_{k^2}(\tau,Y)| \geq a\right) \!.
\end{eqnarray*}
From estimation of $\E[|U_N(\tau,Y)|^2]$ and   applying Tchebychev's
inequality, we have
$$\mathbb{P}(|U_{k^2}(\tau,Y)|          \geq          a)          \leq
\frac{\sigma^2_{\tau}[|\chi|]_{k^2,2}}{a^2 k^4} \leq
\frac{\sigma^2_{\tau}r_{k^2}}{a^2 k^2}\; ,$$
where $r_n= \sum_{j=1}^{n}|\chi(j)|< \|\chi\|_{\ell_1}<\infty$. Then,
$$\sum_{k=1}^{\infty} \frac{r_{k^2}}{k^2} \leq \sum_{k=1}^{\infty} \frac{\|\chi\|_{\ell_1}}{k^2}< \infty.$$
So, the series $\sum_{k\!=\!\lfloor\sqrt{N}\rfloor}^{\infty}\!\!
\mathbb{P}(|U_{k^2}(\tau,Y)|\!\! \geq \!\! a)$ converges.

On the other hand,
$$U_n(\tau,Y)-U_{k^2}(\tau,Y)= \frac{1}{n } \sum_{A(n,k^2)}\Phi_{\tau}(\y^i,\y^j)
+ \left(\frac{1}{n}- \frac{1}{k^2}\right) \sum_{1\leq i\neq j \leq k^2} \Phi_{\tau}(\y^i,\y^j) \; ,$$
where $A(n,k^2)=\{i,j: 1\leq i< j, k^2  < j \leq n\} \cup \{i,j:
1\leq j < i, k^2 <i \leq n \}$. So
\begin{eqnarray*}
\max_{k^2\leq  n \leq  (k+1)^2}|U_n(\tau,Y)-U_{k^2}(\tau,Y)| &  \leq
& a_k + b_k \; ,
\end{eqnarray*}
with
\begin{eqnarray*}
a_k  & = & \max_{k^2\leq n \leq
(k+1)^2}\left|\frac{1}{n}\sum_{A(n,k^2)} \Phi_{\tau}(\y^i,\y^j)\right|
\;  ,\\
b_k  & = & \max_{k^2\leq  n \leq  (k+1)^2}\left| \frac{1}{n}-
\frac{1}{k^2}\right| \left|\sum_{1\leq i<j\leq
k^2} \Phi_{\tau}(\y^i,\y^j) \right| \; .
\end{eqnarray*}
Since $a_k \leq
\frac{1}{k^2}\sum_{A\left(k^2,(k+1)^2\right)}|\Phi_{\tau}(\y^i,\y^j)|$, then
\begin{eqnarray*}
\mathbb{P}(a_k    \geq    a)    &    \leq    &    \frac{1}{a^2 k^4}
\mathbb{E}\left[\left( \sum_{A\left(k^2,(k+1)^2\right)}|
\Psi_{\tau}(\y^i,\y^j)|\right)^2\right]    \\   &
\leq   & \frac{\sigma^2_{\tau} }{a^2 k^4} \left(
\sum_{A\left(k^2,(k+1)^2\right)}|\chi(i-j)| \right)^2 \\
 &    \leq    & \frac{\sigma^2_{\tau} }{a^2 k^4} \left( (k+1)^2 r_{(k+1)^2} - k^2
r_{k^2} \right)^2 \; .
\end{eqnarray*}
Since $\chi$ is a stationary interaction it holds
\begin{eqnarray*}
(k+1)^2 r_{(k+1)^2} - k^2 r_{k^2} & = & \sum_{j=k^2+1}^{(k+1)^2}
\left((k+1)^2 -j\right)|\chi(j)| + \sum_{j=1}^{k^2}\left((k+1)^2 -
k^2\right)|\chi(j)|\\
& \leq & (2k+1) \left(\sum_{j=1}^{(k+1)^2}|\chi(j)|\right)\\
& \leq & 2(k+1)r_{(k+1)^2}.
\end{eqnarray*}
So that
\begin{eqnarray*}
\mathbb{P}(a_k    \geq    a) & \leq & \frac{4\sigma^2_{\tau}(k+1)^2 r^2_{(k+1)^2}}{a^2 k^4}\\
&\leq & \frac{4\sigma^2_{\tau}\|\chi\|^2_{\ell_1} (k+1)^2 }{a^2 k^4}.
\end{eqnarray*}

In the same way, $b_k    \leq   \frac{2}{k^3}\left|\sum_{1\leq
i<j\leq k^2}\Phi_{\tau}(\y^i,\y^j) \right|$ so
$$\mathbb{P}(b_k               \geq               a)              \leq
\frac{4\sigma_{\tau}k^2[|\chi|]_{k^2,2}}{a^2
k^6}\leq \frac{4\sigma_{\tau}\|\chi\|^2_{\ell_1} }{a^2 k^2} \; .$$
Since           $\mathbb{P}(\max_{k^2\leq             n \leq
(k+1)^2}|U_n(\tau,Y)-U_{k^2}(\tau,Y)| \geq a) \leq \mathbb{P}(a_k
\geq a) + \mathbb{P}(b_k \geq a) $ then
\begin{eqnarray*}
\sum_{k=\lfloor\sqrt{N}\rfloor}^{\infty}   \mathbb{P}\left(\max_{k^2\leq
n \leq (k+1)^2}|U_n(\tau,Y)-U_{k^2}(\tau,Y)| \geq a\right) < \infty \; .
\end{eqnarray*}
Finally,  from   Borel-Cantelli's  lemma  we   derive  the
$\nu-a.s.$ convergence to zero of $U_N(\tau,Y)$. This proves the
convergence $\nu-a.s.$

\quad\\
\textbf{Part 3:} ($\Gamma^N$'s convergence). \\
We have proved, in Part 2, that $T_N$, $Q_N$ and $U_N$ converge
$\nu-a.s.$ and in $L_1(\nu)$  to zero. Then, following Part 3 in the proof of Theorem 1 in \cite{Dacunha&Fermin.L.Notes}, we obtain from orthogonal decomposition
$$\Gamma^N(\tau,Y)-R_N(\tau)= T_N(\tau,Y) + Q_N(\tau,Y) + U_N(\tau,Y),$$
that $\Gamma^N(\tau,Y)$ converge in $L^1(\nu)$ and $\nu-a.s.$ to $\Gamma(\tau)$.

\end{proof}


\bibliographystyle{acm}
\bibliography{biblio_ljfnew}

\begin{thebibliography}{10}

\bibitem{Dacunha&Fermin.L.Notes}
{\sc Dacunha-Castelle, D., and Ferm\'{\i}n, L.}
\newblock {Aggregation of Doubly Stochastic Interactive Gaussian Processes and
  Toeplitz forms of $U$-Statistics}.
\newblock {\em {In Dependence in Probability and Statistics, Series: Lecture
  Notes in Statistic.} 187\/} (2006).

\bibitem{Dacunha&Fermin.aggreglinear}
{\sc Dacunha-Castelle, D., and Ferm\'{\i}n, L.}
\newblock {Aggregation of doubly stochastic linear processes}.
\newblock {\em {Preprint.}\/} (2008).

\bibitem{WDB-PDoukhan}
{\sc Dedecker, J., Doukhan, P., Lang, G., Leon, J., Louhichi, S., and Prieur,
  C.}
\newblock {\em {Weak dependence. With examples and applications.}}
\newblock Lecture Notes in Statistics 190. Springer, New York, 2006.

\bibitem{Ding&Granger}
{\sc Ding, Z., and Granger, C.}
\newblock {Modeling volatility persistence of speculative returns: a new
  approach}.
\newblock {\em {J. Econometrics} 52\/} (1996), 185--215.

\bibitem{Doukhan.LRM}
{\sc Doukhan, P.}
\newblock {Models inequalities and limit theorems for stationary sequences}.
\newblock {\em {in Theory and applications of long range dependence (Doukhan et
  al eds.) Birkh\"auser}\/} (2002), 43--101.

\bibitem{Doukhan&Louhichi}
{\sc Doukhan, P., and Louhichi, S.}
\newblock {A new weak dependence condition and applications to moment
  inequaliies }.
\newblock {\em {Stoch. Proc. Appl.} 84\/} (1999), 313--342.

\bibitem{Doukhan&Teyssiere&Winant2006}
{\sc Doukhan, P., Teyssière, G., and Winant, P.}
\newblock {Vector valued ARCH infinity processes. }.
\newblock {\em {In Dependence in Probability and Statistics, Series: Lecture
  Notes in Statistic, Bertail P., Doukhan P. and Soulier P. Eds, Springer, New
  York} 187\/} (2006).

\bibitem{Doukhan.et.al2006}
{\sc Doukhan, P., and Wintenberger, O.}
\newblock {A LARCH infinity vector value processes. }.
\newblock {\em {In Dependence in Probability and Statistics, Series: Lecture
  Notes in Statistic, Bertail P., Doukhan P. and Soulier P. Eds, Springer, New
  York} 187\/} (2006), 245--258.

\bibitem{Doukhan&Wintenberger}
{\sc Doukhan, P., and Wintenberger, O.}
\newblock {An invariance principle for weakly dependent stationary general
  models}.
\newblock {\em {Probab. Math. Stat.} 27}, 1 (2007), 45--73.

\bibitem{Giraitis}
{\sc Giraitis, L., Kokoszka, P., and Leipus, R.}
\newblock {Stationary ARCH models: dependence structure and Central Limit
  Theorem}.
\newblock {\em { Econometric Theory } 16\/} (2000), 3--22.

\bibitem{Giraitis&Surgailis}
{\sc Giraitis, L., and Surgailis}.
\newblock {ARCH-types bilinear models with doubly long memory}.
\newblock {\em { Stochastic Process and their Application } 100\/} (2002),
  275--300.

\bibitem{Goncalvez1988}
{\sc Gon\c{c}alvez, E., and Gourieroux, C.}
\newblock {Agr\'egation de processus autorégressifs d'ordre 1}.
\newblock {\em { Annales d'Economie et de Statistique } 12\/} (1988), 127--149.

\bibitem{Granger}
{\sc Granger, C.}
\newblock {Long Memory relationships and the aggregate of dinamic models}.
\newblock {\em {Journal of Econometrics} 14\/} (1980), 227--238.

\bibitem{Kazakevi&Leipus}
{\sc Kazakevi$\check{c}$ius, V., and Leipus, R.}
\newblock {On the stationary in the $ARCH(\infty)$ model}.
\newblock {\em { Econometric Theory } 18\/} (2002), 1--16.

\bibitem{Kazakevi}
{\sc Kazakevi$\check{c}$ius, V., Leipus, R., and M.C., V.}
\newblock {Stability of random coefficient ARCH models and aggregation
  schemes}.
\newblock {\em { J. of Econometrics} 120\/} (2004), 139--158.

\bibitem{Kokoszka&Leipus}
{\sc Kokoszka, P., and Leipus, R.}
\newblock {Change-point estimation in ARCH models}.
\newblock {\em {Bernoulli.} 6\/} (2000), 513--539.

\bibitem{Leipus&Viano2000}
{\sc Leipus, R., and Viano, M.-C.}
\newblock {Modelling long memory time series with finite or infinite variance:
  a general approach}.
\newblock {\em {Journal of Time Series Analysis} 21}, {1} (2000), 61--67.

\bibitem{Leipus&Viano2002}
{\sc Leipus, R., and Viano, M.-C.}
\newblock {Aggregation in ARCH models}.
\newblock {\em {Liet. math. rink.} 42}, {1} (2002), 68--89.

\bibitem{Linden}
{\sc Linden, M.}
\newblock {Time series properties of aggregated AR(1) processes with uniformly
  distributed coefficients }.
\newblock {\em {Economics Letters} 64\/} (1999), 31--36.

\bibitem{Lippi&Zaffaroni}
{\sc Lippi, M., and Zaffaroni, P.}
\newblock {Aggegation of simple linear dynamics: exact asymptotic results}.
\newblock {\em {Econometrics Discussion Paper 350, STICERD-LSE }\/}.

\bibitem{Louhichi2000}
{\sc Louhichi, S.}
\newblock {Weak convergence for empirical processes of associated sequences}.
\newblock {\em {Ann. Inst. Henri Poincare} 36}, {5} (2000), 547--567.

\bibitem{Neveu}
{\sc Neveu, J.}
\newblock {\em {Processus al\'eatoires gaussiens.}}
\newblock {Seminaire de mathematiques superieures. 34. Montreal, Canada: Les
  Presses de l'Universite de Montreal}, 1968.

\bibitem{Nijman&Sentana}
{\sc Nijman, T., and Sentana, E.}
\newblock {Marginalization and contemporaneous aggregation in multivariate
  GARCH processes}.
\newblock {\em {Journal of Econometrics } 71\/} (1996), 71--87.

\bibitem{Petrov}
{\sc Petrov, V.}
\newblock {\em {Limit Thorems of Probability Theory: Sequences of Independent
  Random Variables}}.
\newblock Oxford Studies in Probability, vol. 4, Oxford Science Publications,
  Clarendon Press, Oxford, 1995.

\bibitem{RobinsonARCH}
{\sc Robinson, P.}
\newblock { Testing for strong serial correlation and dynamic conditional
  heteroscedasticity in multiple regresion}.
\newblock {\em {J. Econometrics} 47\/} (1991), 67--84.

\bibitem{Schreiber}
{\sc Schreiber, M.}
\newblock {Fermeture en probabilit\'e de certains sous-espaces d'un espace
  $L\sp2$. Application aux chaos de Wiener.}
\newblock {\em Z. Wahrscheinlichkeitstheor. Verw. Geb. 14\/} (1969), 36--48.

\bibitem{Surgailis}
{\sc Surgailis, D.}
\newblock {Non CLT's: U-statistics multinomial formula and approximations of
  multiple Ito-Winer integrals}.
\newblock {\em {In Theory and applications of long range dependence (Doukhan et
  al eds.) Birkh\"auser, Boston}\/} (2003), {130--142}.

\bibitem{Terence1}
{\sc Terence, T., and to~W., K.}
\newblock {Time series properties of aggregated AR(2) processes}.
\newblock {\em {Economics Letters} 73\/} (2001), 325--332.

\bibitem{Zaffaroni}
{\sc {Zaffaroni, P.}}
\newblock {Aggregation and memory of models of changing volatility}.
\newblock {\em {Journal of Econometrics} {127}}, {1} (2007), {237--249}.

\end{thebibliography}

\end{document}